\documentclass[12pt]{amsart}
\usepackage[latin1]{inputenc}
\usepackage[T1]{fontenc}
\usepackage{bbm,bm}
\usepackage{amssymb}
\usepackage{mathtools}  
\usepackage{colordvi}
\usepackage{graphicx,xspace}
\usepackage{color}
\usepackage{pstricks}
    \definecolor{mygreen}{rgb}{0,.6,0}
    \definecolor{myblue}{rgb}{0,0,.6}

    \definecolor{frank}{rgb}{1,.2,.2}
    \definecolor{stefan}{rgb}{.2,.7,.2}
    \definecolor{aaron}{rgb}{.2,.2,1}
\usepackage{xspace}

\usepackage[bookmarks=false,colorlinks,linkcolor=myblue,citecolor=mygreen]{hyperref}

\newcounter{FNC}[page]
\def\fauxfootnote#1{{\addtocounter{FNC}{2}$^\fnsymbol{FNC}$%
     \let\thefootnote\relax\footnotetext{$^\fnsymbol{FNC}$\Magenta{#1}}}}

\numberwithin{equation}{section}
\newtheorem{theorem}{Theorem}[section]
\newtheorem{proposition}[theorem]{Proposition}
\newtheorem{lemma}[theorem]{Lemma}
\newtheorem{corollary}[theorem]{Corollary}

\newtheorem{defi}[theorem]{Definition}
\newtheorem{exam}[theorem]{Example}
\newtheorem{rema}[theorem]{Remark}
\newenvironment{definition}[1][]{\rm\begin{defi}[#1]\rm}{\end{defi}}
\newenvironment{example}[1][]{\rm\begin{exam}[#1]\rm}{\end{exam}}
\newenvironment{remark}[1][]{\rm\begin{rema}[#1]\rm}{\end{rema}}

\newpsobject{showgrid}{psgrid}{%
    subgriddiv=1,griddots=10,gridlabels=6pt}
\newif\ifhrule\hrulefalse

\newcommand{\CTO}[1]{\begin{picture}(5.2,11)\put(1.2,0){\includegraphics{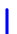}}
    \put(1.2,8.5){\tiny #1}\end{picture}}
\newcommand{\CTI}[2]{\begin{picture}(11.5,12)(-1,0)\put(1.5,0){\includegraphics{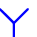}}
    \put(0,9){\tiny #1} \put(8,9){\tiny #2} \end{picture}}
\newcommand{\CTIT}[3]{\begin{picture}(15.5,13)(-1.2,0)\put(1.5,0){\includegraphics{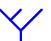}}
    \put(0,10){\tiny #1} \put(6,10){\tiny #2}\put(12,10){\tiny #3} \end{picture}}

\def\treeseparator{\bm\cdot}
\def\treeseparatorinline{\bm\cdot}
\newcommand{\zcompose}[2]{\frac{{#1}}{#2}}
\newcommand{\compose}[3]{\frac{{#1{\treeseparator}\dotsb{\treeseparator}#2}}{#3}}
\newcommand{\lcompose}[4]{\frac{{#1{\treeseparator}#2{\treeseparator}\dotsb{\treeseparator}#3}}{#4}}
\newcommand{\rcompose}[4]{\frac{{#1{\treeseparator}\dotsb{\treeseparator}#2{\treeseparator}#3}}{#4}}

\newcommand{\lrcompose}[5]{\frac{{#1{\treeseparator}#2{\treeseparator}\dotsb{\treeseparator}#3{\treeseparator}#4}}{#5}}
\newcommand{\zcomposeinline}[2]{{{#2}{\,\circ\,}(#1)}}
\newcommand{\composeinline}[3]{{#3{\,\circ\,}({#1{\treeseparatorinline}\dotsb{\treeseparatorinline}#2})}}

\newcommand{\Ddeg}[1]{(#1)}

\def\Cofree{\mathsf{C}}

\newcommand{\demph}[1]{{\it\color{myblue}{#1}}}

\newcommand{\K}{\mathbb{K}}

\newcommand{\calC}{\mathcal{C}}
\newcommand{\calD}{\mathcal{D}}
\newcommand{\calE}{\mathcal{E}}

\def\ssym{\s\mathit{Sym}}
\def\msym{\m\mathit{Sym}}
\def\ysym{\y\mathit{Sym}}

\def\csym{\c\mathit{Sym}}
\def\cksym{\ck\mathit{Sym}}
\def\psym{\mathcal{P}\mathit{Sym}}
\def\delsym{\Delta\mathit{Sym}}

\def\s{\mathfrak S}
\def\m{\mathcal M}
\def\y{\mathcal Y}

\def\ck{{\mathcal C}{\mathcal K}}
\def\c{\mathfrak C}
\def\p{\mathcal P}

\def\btau{\bm\tau}

\def\bkappa{\bm\kappa}

\def\balpha{\bm\alpha}

\newcommand{\sn}[1]{{\small\hspace{2pt}$#1$}}

\newcommand{\setR}{{\sf R}}
\newcommand{\setS}{{\sf S}}
\newcommand{\setT}{{\sf T}}

\def\id{\mathbbm{1}} 

\newcommand{\Placeholder}%
    {\raisebox{.09ex}{\footnotesize $\bullet$}}
\newcommand{\placeholder}%
    {\raisebox{.04ex}{\tiny $\bullet$}}

\def\bb{\boldbullet}
    \def\boldbullet{{\!\mbox{\LARGE$\mathbf\cdot$}}}

\def\psplit{\stackrel{\curlyvee}{\to}}
\def\rsplit{\stackrel{\curlyvee_{\!\!+}}{\longrightarrow}}

\author{Stefan Forcey} \address[S. Forcey]{
    Department of Theoretical and Applied Mathematics\\
    The University of Akron\\
    Akron, OH 44325-4002
    }
    \email{sf34@uakron.edu}  \urladdr{http://www.math.uakron.edu/\~{}sf34/}

\author{Aaron Lauve} \address[A. Lauve]{
    Department of Mathematics\\
    Loyola University of Chicago\\
    Chicago, IL 60660 
    }
    \email{lauve@math.luc.edu}  \urladdr{http://www.math.luc.edu/\~{}lauve/}

    \thanks{Research of Lauve supported in part by NSA grant H98230-10-1-0362.}

\author{Frank Sottile} \address[F. Sottile]{
    Department of Mathematics\\
    Texas A\&M University\\
    College Station, Texas \ 77843 
    }
    \email{sottile@math.tamu.edu}  \urladdr{http://www.math.tamu.edu/\~{}sottile}

    \thanks{Research of Sottile supported in part by NSF grants DMS-0701050 and DMS-1001615.}

\title{Cofree compositions of coalgebras}

\keywords{multiplihedron, composihedron, binary tree, cofree coalgebra, one-sided Hopf algebra, operads, species}
\subjclass[2000]{05E05, 16W30,  18D50}
\begin{document}

\begin{abstract}
 We develop the notion of the composition of two coalgebras, which arises naturally in
 higher category theory and in the theory of species.
 We prove that the composition of two cofree coalgebras is again cofree, and we give
 sufficient conditions that ensure the composition is a one-sided Hopf algebra. 
 We show these conditions are satisfied when
 one coalgebra is a graded Hopf operad $\calD$ and the other is a connected graded
 coalgebra with coalgebra map to $\calD$. We conclude by computing the primitive elements
 for compositions of coalgebras built on the vertices of multiplihedra, composihedra, and
 hypercubes.
\end{abstract}

\maketitle


\section*{Introduction}\label{sec: intro}

The Malvenuto-Reutenauer Hopf algebra of ordered
trees~\cite{MalReu:1995,AguSot:2005} and the Loday-Ronco Hopf
algebra of planar binary trees~\cite{LodRon:1998,AguSot:2006}
are cofree as coalgebras and are connected by cellular maps from the
permutahedra to the associahedra. Closely related polytopes include
Stasheff's multiplihedra~\cite{Sta:1970} and the
composihedra~\cite{forcey2}, and it is natural to study to what
extent Hopf structures may be placed on these objects. The map from
permutahedra to associahedra factors through the multiplihedra, and
in~\cite{FLS:2010} we used this factorization to place Hopf structures
on bi-leveled trees, which correspond to vertices of multiplihedra.

The multiplihedra form an operad module over the associahedra, and
this leads to the concept of painted trees, which also correspond to
the vertices of the multiplihedra. Moreover, expressing the Hopf
structures of~\cite{FLS:2010} in terms of painted trees relates these
Hopf structures to the operad module structure. Abstracting this
structure leads to the general notion of a composition of coalgebras,
which is a functorial construction of a graded coalgebra
$\calD\circ\calC$ from graded coalgebras $\calC$ and $\calD$.
We define this composition in Section \ref{sec: cccc} and show that it
preserves cofreeness. In Section \ref{sec: hopf}, we suppose that $\calD$
is a Hopf algebra and give sufficient conditions for the compositions
of coalgebras $\calD\circ\calC$ and $\calC\circ\calD$ to be one-sided Hopf algebras.
These also guarantee that these compositions are Hopf modules and comodule
algebras over $\calD$.

The definition of the composition of coalgebras is
familiar from the theory of operads. In general, a (nonsymmetric)
operad is a monoid in the category of graded sets, with product given by composition
(also known as the substitution product).
In Section \ref{sec: hopf} we show that an operad $\calD$ in
the category of connected graded coalgebras is automatically a Hopf algebra.
Those familiar with the theory of species will also recognize our construction.
The coincidence is explained in~\cite[Appendix~B]{AguMah:2010}: species and operads are
one-and-the-same.

 We conclude in Sections \ref{sec: painted}, \ref{sec: weighted}, and~\ref{sec: simplices}
 with a detailed look at several compositions of coalgebras that
enrich the understanding of well-known objects from category theory and algebraic topology.
In particular, we prove that the (one sided) Hopf algebra of simplices
in~\cite{ForSpr:2010} is cofree as a coalgebra.


\section{Preliminaries}\label{sec: prelims}

We work over a fixed field $\K$ of characteristic zero.
For a graded vector space $V = \bigoplus_n V_n$, we write $|v|=n$ and say $v$ has
\demph{degree} $n$ if $v\in V_n$.

\subsection{Hopf algebras and cofree coalgebras}\label{sec: cofree def}

A bialgebra $H$ is a unital associative algebra equipped with two algebra maps: a coproduct
homomorphism $\Delta\colon H\to  H\otimes H$ that is coassociative and a counit
homorphism $\varepsilon\colon H\to \K$ which plays the role of the identity for $\Delta$.
See \cite{Mont:1993} for more details.
A graded bialgebra
$H=(\bigoplus_{n\geq0}H_n,\bm\cdot,\Delta,\varepsilon)$ is \demph{connected} if
$H_0 = \K$.
In this case, a result of Takeuchi~\cite[Lemma~14]{Tak:71} guarantees the existence of an
antipode map for $H$, making it a Hopf algebra.

We recall Sweedler's coproduct notation for later use.
A coalgebra $\calC$ is a vector space $\calC$ equipped with a coproduct $\Delta$
and counit $\varepsilon$.
Given $c\in\calC$, the coproduct $\Delta(c)$ is written $\sum_{(c)} c'\otimes c''$.
Coassociativity means that
\[
    \sum_{(c),(c')} (c')'\otimes (c')'' \otimes c''\ =\
    \sum_{(c),(c'')} c'\otimes (c'')' \otimes (c'')''\ =\
    \sum_{(c)} c'\otimes c'' \otimes c''' \,,
\]
and the counit condition means that
$\sum_{(c)} \varepsilon(c')c'' = \sum_{(c)} c'\varepsilon(c'') = c$.

The \demph{cofree coalgebra} on a vector space $V$ has underlying vector space
$\Cofree(V):= \bigoplus_{n\geq0}V^{\otimes n}$.
Its counit is the projection $\varepsilon\colon\Cofree(V)\to\K=V^{\otimes 0}$.
Its coproduct is the \demph{deconcatenation coproduct}: writing ``$\backslash$''
for the tensor product in $V^{\otimes n}$, we have
\[
    \Delta(c_1\backslash \dotsb\backslash c_n)\ =\
    \sum_{i=0}^{n} (c_1\backslash \dotsb\backslash c_{i})
        \otimes (c_{i+1}\backslash \dotsb\backslash c_n) \,.
\]
Observe that $V$ is exactly the set of primitive elements of $\Cofree(V)$.
A coalgebra $\calC$ is \demph{cofreely cogenerated} by a subspace $V\subset\calC$ if
$\calC\simeq\Cofree(V)$ as coalgebras.
Necessarily, $V$ is the space of primitive elements of $\calC$.
Many of the coalgebras and Hopf algebras arising in combinatorics are cofree.
We recall a few key examples.

\subsection{Cofree Hopf algebras on trees}\label{sec: Hopf examples}

We describe three cofree Hopf algebras built on rooted planar binary
trees:
\demph{ordered trees} $\s_n$, \demph{binary trees} $\y_n$, and \demph{(left) combs} $\c_n$ on
$n$ internal nodes.
Let $\s_\bb$ denote the union $\bigcup_{n\geq0} \s_n$ and define $\y_\bb$ and $\c_\bb$
similarly.

\subsubsection{Constructions on trees}\label{sec: trees}

The nodes of a tree $t\in\y_n$ are a poset (with root maximal) whose Hasse diagram is the
internal edges of $t$.
An \demph{ordered tree $w=w(t)$} is a linear extension of this node poset of $t$ that we
indicate by placing a permutation in the gaps between its leaves.
Ordered trees are in bijection with the permutations of $n$.
The map $\tau \colon \s_n \to \y_n$ forgets the total ordering of the nodes of
an ordered tree $w(t)$ and gives the underlying tree $t$.
The map $\kappa \colon \y_n \to \c_n$ shifts all nodes of a tree $t$ to the right
branch from the root.
We let $\s_0=\y_0=\c_0=\includegraphics{figures/0.eps}$.
Note that $|\c_n|=1$ for all $n\geq0$.

Figure \ref{fig: many trees} gives some examples from $\s_\bb$, $\y_\bb$, and $\c_\bb$ and
indicates the natural maps $\tau$ and $\kappa$ between them.
See \cite{FLS:2010} for more details.
\begin{figure}[htb]
\[
\psset{unit=10pt}
\begin{pspicture}(8.5,20)
   \thicklines
   \put(.5,0){\small ordered trees $\s_\bb$}

   \put(4,18.1){%
         \begin{pspicture}(1,2)
         \put(0,0){\includegraphics{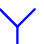}}
         \put(0,1.1){\sn{1}}
      \end{pspicture}}

   \put(3.5,14.0){%
      \begin{pspicture}(2,2.5)
         \put(0,0){\includegraphics{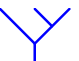}}
         \put(0,1.7){\sn{2}} \put(1,1.7){\sn{1}}
      \end{pspicture}}
   \put(3.5,10.7){%
      \begin{pspicture}(2,2.5)
         \put(0,0){\includegraphics{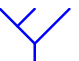}}
         \put(0,1.7){\sn{1}} \put(1,1.7){\sn{2}}
      \end{pspicture}}

   \put(0,6.0){%
      \begin{pspicture}(4,3.5)
         \put(0,0){\includegraphics{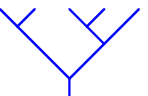}}
      \put(0,2.7){\sn{3}} \put(1.0,2.7){\sn{4}} \put(2.0,2.7){\sn{1}} \put(3.0,2.7){\sn{2}}
      \end{pspicture}}
   \put(4.5,3.7){%
      \begin{pspicture}(4,3.5)
         \put(0,0){\includegraphics{figures/3412.d.eps}}
      \put(0,2.7){\sn{2}} \put(1.0,2.7){\sn{4}} \put(2.0,2.7){\sn{1}} \put(3.0,2.7){\sn{3}}
      \end{pspicture}}
   \put(0,1.5){%
      \begin{pspicture}(4,3.5)
         \put(0,0){\includegraphics{figures/3412.d.eps}}
      \put(0,2.7){\sn{1}} \put(1.0,2.7){\sn{4}} \put(2.0,2.7){\sn{2}} \put(3.0,2.7){\sn{3}}
      \end{pspicture}}

\end{pspicture}
\qquad
\begin{pspicture}(3,20)
   \thicklines
   \put(0,0){\vector(1,0){3.0}}\put(1,0.4){$\tau$}

   \put(0,18.6){\Color{1 0 1 .3}{\vector(1,0){3.0}}}
   \put(0,14.9){\Color{1 0 1 .3}{\vector(1,0){3.0}}}
   \put(0,11.6){\Color{1 0 1 .3}{\vector(1,0){3.0}}}
   \put(0, 5.2){\Color{1 0 1 .3}{\vector(1,0){3.0}}}
\end{pspicture}
\qquad
\begin{pspicture}(6.5,20)
   \thicklines
   \put(0,0){\small binary trees $\y_\bb$}

   \put(3,18.1){%
         \includegraphics{figures/1.d.eps}}

   \put(2.5,14.0){\includegraphics{figures/21.d.eps}}
   \put(2.5,10.7){\includegraphics{figures/12.d.eps}}

   \put(1.5,3.8){%
      \includegraphics{figures/3412.d.eps}}

\end{pspicture}
\qquad
\begin{pspicture}(3,20)
   \thicklines
   \put(0,0){\vector(1,0){3.0}}\put(1,0.4){$\kappa$}

   \put(0,18.6){\Color{1 0 1 .3}{\vector(1,0){3.0}}}
   \put(0,13.3){\Color{1 0 1 .3}{\vector(1,0){3.0}}}
   \put(0, 5.2){\Color{1 0 1 .3}{\vector(1,0){3.0}}}
\end{pspicture}
\qquad
\begin{pspicture}(6.5,20)
   \thicklines
   \put(0,0){\small left combs $\c_\bb$}

   \put(2,18.1){\includegraphics{figures/1.d.eps}}

   \put(1.5,12.4){\includegraphics{figures/21.d.eps}}

   \put(0.5,3.8){\includegraphics{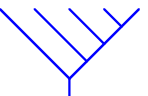}}

\end{pspicture}
\]
\caption{Maps between binary trees.}
\label{fig: many trees}
\end{figure}

%

\demph{Splitting} an ordered tree $w$ along the path from a leaf to the root yields an ordered
forest (where the nodes in the forest are totally ordered) or a pair of ordered trees,
\[
   \begin{picture}(52,49)
   \put(0,0){\includegraphics{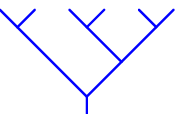}}
   \put(0,31){\sn{2}} \put(10,31){\sn{5}} \put(20,31){\sn{1}}
   \put(30,31){\sn{4}}\put(40,31){\sn{3}}
   \put(30,49){\vector(0,-1){15}}
  \end{picture}
   \ \raisebox{12pt}{\large$\leadsto$}\
  \begin{picture}(52,49)
   \put(0,0){\includegraphics{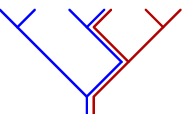}}
   \put(0,31){\sn{2}} \put(10,31){\sn{5}} \put(20,31){\sn{1}}
   \put(32,31){\sn{4}}\put(42,31){\sn{3}}
   \put(31,49){\vector(0,-1){15}}
  \end{picture}
   \ \raisebox{12pt}{$\xrightarrow{\ \curlyvee\ }$}\quad
 \raisebox{7.5pt}{$\displaystyle
   \raisebox{8pt}{$\biggl(\;$}
   \begin{picture}(32,27)
    \put(0,0){\includegraphics{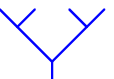}}
    \put(0,21){\sn{2}} \put(10,21){\sn{5}} \put(20,21){\sn{1}}
   \end{picture}
     ,\
   \begin{picture}(22,23)
    \put(0,0){\includegraphics{figures/21.d.eps}}
    \put(0,16){\sn{4}} \put(10,16){\sn{3}}
   \end{picture}
   \raisebox{8pt}{$\;\biggr)$}
   \quad \raisebox{5.5pt}{\mbox{or}}\quad
   \raisebox{8pt}{$\biggl(\;$}
   \begin{picture}(32,27)
    \put(0,0){\includegraphics{figures/231.d.eps}}
    \put(0,21){\sn{2}} \put(10,21){\sn{3}} \put(20,21){\sn{1}}
   \end{picture}
     ,\
   \begin{picture}(22,23)
    \put(0,0){\includegraphics{figures/21.d.eps}}
    \put(0,16){\sn{2}} \put(10,16){\sn{1}}
   \end{picture}
   \raisebox{8pt}{$\;\biggr)$}.
    $}
\]
Write $w\psplit(w_0,w_1)$ when the ordered forest $(w_0,w_1)$
(or pair of ordered trees) is obtained by splitting $w$.
(Context will determine how to interpret the result.)

We may \demph{graft} an ordered forest $\vec w = (w_0,\dotsc,w_n)$ onto an ordered tree
$v\in\s_n$, obtaining the tree $\vec w/v$ as follows.
First increase each label of $v$ so that its nodes are greater than the nodes of $\vec w$, and
then graft tree $w_i$ onto the $i^{\mathrm{th}}$ leaf of $v$.
For example,
\begin{align*}
\hbox{if } (\vec w, v)  &\ =\
  \raisebox{0pt}{$\displaystyle
   \raisebox{8pt}{$\Biggl( \biggl(\;$}
       \begin{picture}(22,23)
          \put(0,0){\includegraphics{figures/21.d.eps}}
          \put(0,16){\sn{3}} \put(10,16){\sn{2}}
       \end{picture}
         , \
         \includegraphics{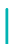}
        \ ,
       \begin{picture}(32,28)
          \put(0,0){\includegraphics{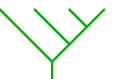}}
              \put(0,21){\sn{7}} \put(10,21){\sn{5}}\put(20,21){\sn{1}}
       \end{picture}
         , \
       \begin{picture}(12,18)
           \put(0,0){\includegraphics{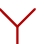}}
           \put(0,11){\sn{6}}
       \end{picture}
        \ , \
       \begin{picture}(12,18)
           \put(0,0){\includegraphics{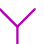}}
           \put(0,11){\sn{4}}
       \end{picture}
   \raisebox{8pt}{$\;\biggr)$} ,
       \begin{picture}(42,35)(-2,0)
              \put(0,0){\includegraphics{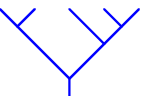}}
              \put(0,27){\sn{1}} \put(10,27){\sn{4}} \put(20,27){\sn{3}} \put(30,27){\sn{2}}
            \end{picture}
   \raisebox{8pt}{$\;\Biggr)$},
   $}
\\[2ex]
\raisebox{30pt}{then $\vec w/v$} & \ \  \raisebox{30pt}{$=$} \ \
  \begin{picture}(125,88)
    \put(0,0){\includegraphics{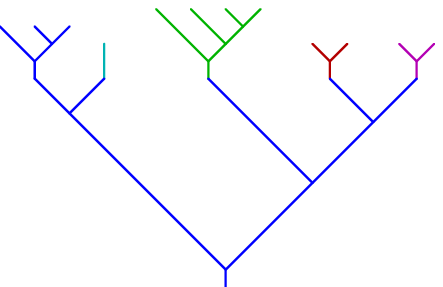}}
    \put(0,76){\sn{3}} \put(10,76){\sn{2}} \put(15,61){\sn{8}} \put(36,61){\sn{11}}
    \put(45,81){\sn{7}} \put(55,81){\sn{5}} \put(65,81){\sn{1}} \put(71,61){\sn{10}}
    \put(90,71){\sn{6}} \put(103,61){\sn{9}} \put(115,71){\sn{4}}
  \end{picture}
  \quad\raisebox{30pt}{$=$}\quad
  \begin{picture}(120,73)
    \put(0,0){\includegraphics{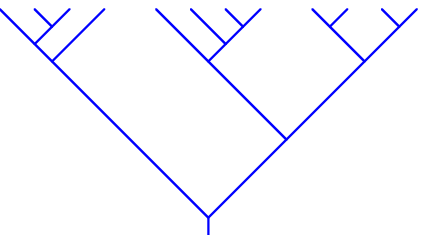}}
    \put(0,66){\sn{3}} \put(10,66){\sn{2}} \put(20,66){\sn{8}} \put(30,66){\sn{11}}
    \put(45,66){\sn{7}} \put(55,66){\sn{5}} \put(65,66){\sn{1}} \put(75,66){\sn{10}}
    \put(90,66){\sn{6}} \put(101,66){\sn{9}} \put(110,66){\sn{4}}
  \end{picture}
\raisebox{30pt}{.}
\end{align*}

The notions of splitting and grafting make sense for trees in $\y_\bb$ (simply forget the
labels on the nodes).
They also work for $\c_\bb$, if after grafting a forest of combs
onto the leaves of a comb, $\kappa$ is applied to the resulting planar binary tree to get
a new comb.

\subsubsection{Three cofree Hopf algebras}\label{sec: cofree examples}

Let $\ssym := \bigoplus_{n\geq0} \ssym_n$ be the graded vector space whose $n^{\mathrm{th}}$
graded piece has basis $\{F_w \mid w \in \s_n\}$.
Define $\ysym$ and $\csym$ similarly.
The set maps $\tau$ and $\kappa$ induce vector space maps $\btau$ and $\bkappa$,
$\btau(F_w) = F_{\tau(w)}$ and $\bkappa(F_t)=F_{\kappa(t)}$.
Fix $\mathfrak X \in \{\s, \y, \c\}$.
For $w\in \mathfrak X_\bb$ and $v\in \mathfrak X_n$, set
 \begin{align*} 
    F_w \cdot F_v\ &:=\ \sum_{w\psplit(w_0,\dotsc,w_n)} F_{(w_0,\dotsc,w_n)/v} \,,
  \intertext{the sum over all ordered forests obtained by splitting $w$ at a multiset of $n$
  leaves, and }
    \Delta(F_w) \ &:= \ \sum_{w\psplit(w_0,w_1)} F_{w_0} \otimes F_{w_1} \,,
 \end{align*}
the sum over pairs of trees obtained by splitting $w$ at one leaf.
The counit $\varepsilon$ is projection onto the $0^\mathrm{th}$ graded piece, which is spanned by the unit element $1=F_{\includegraphics{figures/0.eps}}$ for the multiplication.

\begin{proposition}\label{thm: Hopf examples}
 For $(\Delta,\cdot,\varepsilon)$ above,
 $\ssym$ is the Malvenuto--Reutenauer Hopf algebra of permutations,
 $\ysym$ is the Loday--Ronco Hopf algebra of planar binary trees, and
 $\csym$ is the divided power Hopf algebra.
 Moreover, $\ssym\xrightarrow{\,\btau\,}\ysym$ and $\ysym\xrightarrow{\,\bkappa\,}\csym$ are
 surjective Hopf algebra maps. \hfill \qed
\end{proposition}

The part of the proposition involving $\ssym$ and  $\ysym$ is found in~\cite{AguSot:2005,AguSot:2006};
the part involving $\csym$ is straightforward and we leave it to the reader.

\begin{remark}
  Typically~\cite[Example~5.6.8]{Mont:1993}, the divided power Hopf algebra is defined to be
 $\K[x] := \mathrm{span}\{x^{(n)} \mid n\geq0\}$,
 with basis vectors $x^{(n)}$ satisfying
  $x^{(m)}\cdot x^{(n)} = \binom{m+n}{n} x^{(m+n)}$,
  $1=x^{(0)}$, $\Delta(x^{(n)}) = \sum_{i+j=n} x^{(i)} \otimes x^{(j)}$,
 and $\varepsilon(x^{(n)})=0$ for $n>0$.
 An  isomorphism between $\K[x]$ and $\csym$ is given by
  $x^{(n)}\mapsto F_{c_n}$, where $c_n$ is the
 unique comb in $\c_n$.
\end{remark}

The following result is important for what follows.

\begin{proposition}\label{thm: cofree examples}
 The Hopf algebras $\ssym$, $\ysym$, and $\csym$ are cofreely cogenerated by their primitive
 elements. \qed
\end{proposition}

The result for $\csym$ is easy.
Proposition~\ref{thm: cofree examples} is proven for $\ssym$ and $\ysym$ in~\cite{AguSot:2005}
and \cite{AguSot:2006} by performing a change of basis---from the \demph{fundamental basis}
$F_w$ to the \demph{monomial basis} $M_w$---by means of M\"obius inversion in a poset structure
placed on $\s_\bb$ and $\y_\bb$.
We revisit this in Section \ref{sec: coalgebra psym}.

\newpage
\section{Cofree Compositions of Coalgebras}\label{sec: cccc}

\subsection{Cofree composition of coalgebras}\label{sec: cccc main results}

Let $\calC$ and $\calD$ be graded coalgebras.
We form a new coalgebra $\calE=\calD\circ\calC$ on the vector space
 \begin{gather}\label{eq: E_(n)}
  \calD\circ\calC \ := \ \bigoplus_{n\geq0} \calD_n \otimes \calC^{\otimes(n+1)} \,.
 \end{gather}
We write $\calE=\bigoplus_{n\geq0}\calE_{\Ddeg{n}}$, where
$\calE_{\Ddeg{n}}= \calD_n \otimes \calC^{\otimes(n+1)}$.
This gives a coarse coalgebra grading of $\calE$ by \demph{$\calD$-degree}.
There is a finer grading of $\calE$ by \demph{total degree}, in which a decomposable
tensor  $c_0\otimes \dotsb\otimes c_n \otimes d$ (with $d\in\calD_n$)
has total degree $|c_0|+\dotsb+|c_n|+|d|$.
Write $\calE_n$ for the linear span of elements of total degree $n$.
%
%

%
%

\begin{example}\label{ex: painted}
 This composition is motivated by a grafting construction on trees.  Let
 $d\times (c_0,\dotsc,c_n) \in {\red\y_n} \times \bigl({\blue\y_\bb}^{{n+1}}\bigr)$.
 Define $\circ$ by attaching the forest $(c_0,\dots,c_n)$ to the leaves of $d$
 while remembering $d$,
 \[
   \raisebox{0pt}{$\displaystyle
      \includegraphics{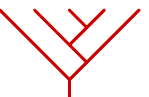}
\raisebox{2pt}{\,\ $\times$\, }
    \raisebox{8pt}{$\biggl(\;$}
      \includegraphics{figures/12.d.eps}
     , \
      \includegraphics{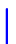}
      \ ,
      \includegraphics{figures/231.d.eps}
     , \
      \includegraphics{figures/1.d.eps}
         \ , \
      \includegraphics{figures/0.d.eps}
   \raisebox{8pt}{$\;\biggr)$}
   $} \raisebox{3pt}{\ \ \ $\stackrel{\textstyle\circ}{\longmapsto}$\ }
   \raisebox{-18pt}{\includegraphics{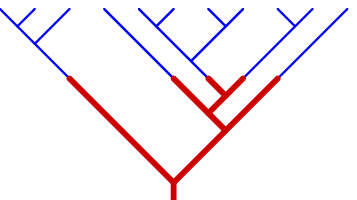}} .
 \]
 This is a new type of tree (\demph{painted trees} in Section \ref{sec: painted}).
 Applying this construction to the indices of basis elements of $\calC$ and $\calD$ and
 extending by multilinearity gives $\calD \circ \calC$.

\end{example}

Motivated by this example, we represent an decomposable tensor in
$\calD\circ\calC$ as
\[
    \composeinline{c_0}{c_n}{d} \qquad\hbox{or}\qquad \compose{c_0}{c_n}{d}
\]
to compactify notation.

\subsubsection{The coalgebra $\calD\circ\calC$}
We define the \demph{compositional coproduct} $\Delta$ for $\calD\circ\calC$ on indecomposable
tensors and extend multilinearly: if $|d|=n$, put
 \begin{equation}\label{eq: delta}
    \Delta\left(\compose{c_0}{c_n}{d}\right)\ =\
      \sum_{i=0}^n \sum_{\substack{(d) \\ |d'|=i}}
    \sum_{(c_i)} \rcompose{c_0}{c_{i-1}}{c'_i}{d'} \otimes \lcompose{c''_i}{c_{i+1}}{c_n}{d''} \,,
 \end{equation}
where the coproducts in $\calC$ and $\calD$ are expressed using Sweedler notation.

The \demph{counit} $\varepsilon:\calD\circ\calC \to \K$ is given by
$\varepsilon(\composeinline{c_0}{c_n}{d})=
  \varepsilon_\calD(d) \cdot \prod_j \varepsilon_\calC(c_j) $.
Hence, it is zero off of $\calD_0 \otimes \calC_0$.

\begin{remark}
 The reader may check that for the painted trees of Example~\ref{ex: painted},
 if $c_0,\dotsc,c_n$ and $d$ are elements of the $F$-basis of $\ysym$, then 
 ${\composeinline{c_0}{c_n}{d}}$ represents a painted tree $t$ and 
 $\Delta ({\composeinline{c_0}{c_n}{d}})$ is the sum over all splittings 
 $t\psplit(t',t'')$ of $t$ into a pair of painted trees.
\end{remark}

\begin{theorem}
  $(\calD\circ\calC,\Delta,\varepsilon)$ is  a coalgebra.
  This composition is functorial, if $\varphi\colon\calC\to\calC'$ and
  $\psi\colon\calD\to\calD'$ are morphisms of graded coalgebras, then
\[
    \compose{c_0}{c_n}{d}\ \longmapsto\
    \compose{\varphi(c_0)}{\varphi(c_n)}{\psi(d)}
\]
  defines a morphism of graded coalgebras
  $\varphi\circ\psi\colon \calD\circ\calC\to \calD'\circ\calC'$.
\end{theorem}

\begin{proof}
 Let $\id$ be the identity map.
 Fix $e:=\composeinline{c_0}{c_n}{d} \in (\calD\circ\calC)_{\Ddeg{n}}$. From \eqref{eq: delta}, we have
 \begin{align*}
   (\Delta\otimes \id)\Delta (e) \ =\ &
    \sum_{i=0}^n \sum_{j=0}^{i-1} \sum_{\substack{(d),(d') \\ |d'|=i,|(d')'|=j}}
    \sum_{(c_i),(c_j)} \compose{c_0}{c'_j}{(d')'} \otimes \compose{c''_j}{c'_i}{(d')''}
    \otimes \compose{c''_i}{c_n}{d''} \\
&\mbox{\ \ }\quad + \sum_{i=0}^n \sum_{\substack{(d),(d') \\ |d'|=i,|(d')''|=0}}
    \sum_{(c_i),(c'_i)} \compose{c_0}{(c'_i)'}{(d')'} \otimes \zcompose{(c'_i)''}{(d')''}
    \otimes \compose{c''_i}{c_n}{d''} \ .
\intertext{Using coassociativity, this becomes}
  &    \sum_{i=0}^n \sum_{j=0}^{i-1} \sum_{\substack{(d) \\ |d'|=i,|d''|=j}}
    \sum_{(c_i),(c_j)} \compose{c_0}{c'_j}{d'} \otimes \compose{c''_j}{c'_i}{d''}
    \otimes \compose{c''_i}{c_n}{d'''} \\
  &\mbox{\ \ }\quad+ \sum_{i=0}^n \sum_{\substack{(d) \\ |d'|=i,|d''|=0}}
    \sum_{(c_i)} \compose{c_0}{c'_i}{d'} \otimes \zcompose{c''_i}{d''}
    \otimes \compose{c'''_i}{c_n}{d'''} \ .
\end{align*}
Simplification of $(\id\otimes\Delta)\Delta(e)$ reaches the same expression, proving
coassociativity.

For the counital condition, we have
\begin{align*}
(\varepsilon\otimes \id)\Delta (e) \ =\ &
    \sum_{i=0}^n \sum_{\substack{(d) \\ |d'|=i}}
    \sum_{(c_i)} \varepsilon\!\left(\rcompose{c_0}{c_{i-1}}{c'_i}{d'}\right)
    \lcompose{c''_i}{c_{i+1}}{c_n}{d''} \\
\ =\ &
    \sum_{\substack{(d) \\ |d'|=0}}
    \sum_{(c_0)} \varepsilon\!\left(\zcompose{c'_0}{d'}\right)
    \lcompose{c''_0}{c_{1}}{c_n}{d''} \,, \\
\intertext{since $\varepsilon_\calD(d')=0$ unless $|d'|=0$. Continuing, this becomes}
&  \sum_{\substack{(d) \\ |d'|=0}}
    \sum_{(c_0)} \lcompose{\varepsilon(c'_0)c''_0}{c_{1}}{c_n}{\varepsilon(d')d''} \ = \ e\,,
\end{align*}
by the counital conditions in $\calC$ and $\calD$.
The identity $(\id \otimes \varepsilon)\Delta(e) = e$ is similarly verified, proving the
counital condition for $\calD\circ\calC$.
Lastly, the functoriality is clear.
\end{proof}

%

\subsubsection{The cofree coalgebra $\calD\circ\calC$}
Suppose that $\calC$ and $\calD$ are graded, connected, and cofree.
Then $\calC=\Cofree(P_\calC)$, where $P_\calC\subset\calC$ is its space of primitive elements.
Likewise, $\calD=\Cofree(P_\calD)$.
As in Section~\ref{sec: cofree def}, we use ``$\backslash$''  for internal tensor products.

\begin{theorem}\label{thm: cc cofree}
 If\/ $\calC$ and $\calD$ are cofree coalgebras  then $\calD\circ\calC$ is also a cofree
 coalgebra.
 Its space of  primitive  elements is spanned by indecomposible tensors of the form
\begin{equation}\label{eq: primitives}
    \lrcompose{1}{c_1}{c_{n-1}}{1}{\delta} \qquad\hbox{ and }\qquad \zcompose{\,\gamma\,}{1} \,
\end{equation}
 where $\gamma,c_i\in\calC$ and $\delta\in\calD_n$ with $\gamma,\delta$ primitive.
\end{theorem}

\begin{proof}
Let $\calE = \calD\circ\calC$ and let $P_\calE$ denote the vector space spanned by the vectors
in \eqref{eq: primitives}.
We compare the compositional coproduct $\Delta$ to the deconcatenation coproduct $\Delta_{\Cofree}$ on
the space $\Cofree(P_\calE)$.
We define a vector space isomorphism $\varphi\colon \calE \to \Cofree(P_\calE)$ and check that
$\Delta_{\Cofree}\,\varphi(e) = (\varphi\otimes\varphi)\Delta(e)$ for all $e\in\calE$.

Let $e = \composeinline{c_0}{c_n}{d}$.
Define $\varphi$ recursively as follows:
\begin{itemize}
  \item If $d=1$ and $c_0 = c'_{0}\backslash c''_{0}$, put
    $\displaystyle \varphi\left(\zcompose{c_{0}}{1}\right) =
    \varphi\left(\zcompose{c'_{0}}{1}\right) \Big\backslash
        \varphi\left(\zcompose{c''_{0}}{1}\right)$.

\item If $|c_0|>0$, put $\displaystyle \varphi(e) = \varphi\left(\zcompose{c_0}{1}\right) \Big\backslash \varphi\left(\lcompose{1}{c_1}{c_n}{d}\right)$.

\item If $|c_n|>0$, put $\displaystyle \varphi(e) =  \varphi\left(\rcompose{c_0}{c_{n-1}}{1}{d}\right)\Big\backslash \varphi\left(\zcompose{c_n}{1}\right)$.

\item If $d=d'\backslash d''$ with $|d'|=i$, then put
$\displaystyle \varphi(e) = \varphi\left(\compose{c_0}{c_i}{d'}\right) \Big\backslash \varphi\left(\lcompose{1}{c_{i+1}}{c_n}{d''}\right)$.
\end{itemize}
We illustrate $\varphi$ with an example from $\calE_{\Ddeg{5}}$:
\[
\zcompose{a'\backslash a'' \treeseparator b \treeseparator c \treeseparator u'\backslash u'' \treeseparator v \treeseparator w}{d'\backslash d''}
    \ \ \stackrel{\varphi}{\longmapsto} \  \
\zcompose{a'}{1} \,\Big\backslash\, \zcompose{a''}{1} \,\Big\backslash\,
\zcompose{1 \treeseparator b \treeseparator c \treeseparator 1}{d'} \,\Big\backslash\,
\zcompose{u'}{1} \,\Big\backslash\, \zcompose{u''}{1} \,\Big\backslash\,
\zcompose{1 \treeseparator v \treeseparator 1}{d''} \,\Big\backslash\, \zcompose{w}{1} \,.
\]
Here $|d'|=3$ and all variables belong to $P_\calC \cup P_\calD$.

To see that $\varphi$ is a coalgebra map, notice that locations to deconcatenate $\varphi(e)$,
\[
   t_1 \backslash \dotsb \backslash t_N \ \longmapsto \ 
   t_1 \backslash \dotsb \backslash t_i \otimes t_{i_1} \backslash \dotsb \backslash t_N \,,
\]
are in bijection with pairs of locations: a place to deconcatenate $d$ and a place to
deconcatenate an accompanying $c_i$.
These are exactly the choices governing \eqref{eq: delta}, given that $d$ and each $c_i$ belong
to tensor powers of $P_\calD$ and $P_\calC$, respectively.
\end{proof}

\subsection{Examples of cofree compositions of coalgebras}\label{sec: cccc examples}

The graded Hopf algebras of ordered trees $\ssym$, planar trees $\ysym$, and divided
powers $\csym$ are all cofree, and so their compositions are cofree.
We have the surjective Hopf algebra maps
\[
   \ssym\ \longrightarrow\ \ysym\ \longrightarrow\ \csym
\]
giving a commutative diagram of nine cofree coalgebras
(Figure \ref{fig: diamond}), as the composition $\circ$ is functorial.
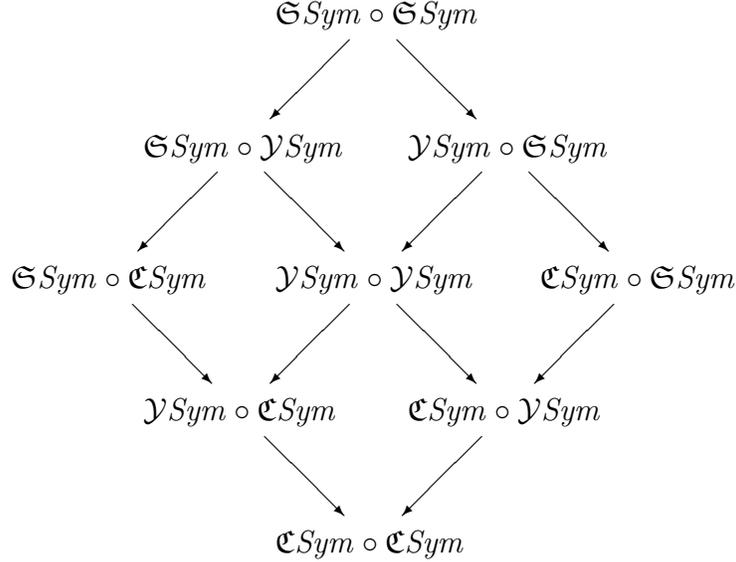
\begin{figure}[!hbt]
\centering
  \begin{picture}(270,212)
                       \put(100,200){$\ssym\circ\ssym$}
    \put(128,194){\vector(-1,-1){30}}     \put(146,194){\vector(1,-1){30}}
       \put(50,150){$\ssym\circ\ysym$}   \put(150,150){$\ysym\circ\ssym$}
    \put( 96,144){\vector(1,-1){30}}     \put(178,144){\vector(-1,-1){30}}
    \put( 78,144){\vector(-1,-1){30}}      \put(196,144){\vector(1,-1){30}}
   \put(0,100){$\ssym\circ\csym$}        \put(200,100){$\csym\circ\ssym$}
                    \put(100,100){$\ysym\circ\ysym$}
    \put( 46,94){\vector(1,-1){30}}     \put(128,94){\vector(-1,-1){30}}
    \put(146,94){\vector(1,-1){30}}     \put(228,94){\vector(-1,-1){30}}
        \put(50,50){$\ysym\circ\csym$}   \put(150,50){$\csym\circ\ysym$}
    \put( 96,44){\vector(1,-1){30}}     \put(178,44){\vector(-1,-1){30}}
                  \put(100,0){$\csym\circ\csym$}
  \end{picture}
\caption{A commutative diagram of cofree compositions of coalgebras.}\label{fig: diamond}
\end{figure}
In Section \ref{sec: hopf}, we use operads to further analyze eight of these nine (all
except $\ssym \circ \ssym$). 
We show that these eight are one-sided Hopf algebras. The algebra $\psym$ of
painted trees appears in the center of this $3\times 3$ grid. We
discuss $\psym$ further in Section \ref{sec: painted}, the algebra
$\ysym\circ\csym$ in Section~\ref{sec: weighted}, and the algebra
$\csym\circ\csym$ in Section~\ref{sec: simplices}.

\subsection{Enumeration}\label{S:enumeration}
We enumerate the graded dimension of many examples from Section~\ref{sec: cccc examples}.
Set $\calE:=\calD\circ\calC$ and let $C_n$ and $E_n$ be the dimensions of
$\calC_n$ and $\calE_n$, respectively.

\begin{theorem}
 When $\calD_n$ has a basis indexed by combs with $n$ internal nodes
 we have the recursion
 \[
   E_0\ =\ 1\,,\quad\mbox{and for \ $n>0$,}\quad
  E_n\ =\ C_n\ +\ \sum_{i=0}^{n-1}C_i E_{n-i-1}\,.
 \]
\end{theorem}

\begin{proof}
 The first term in the expression for $E_n$ counts elements  in $\calE_n$ of the form
 $\includegraphics{figures/0.eps}\circ c$.
 Removing the root node of $d$ from ${\composeinline{c_0}{c_k}{d}}$ gives a
 pair  $\zcomposeinline{c_0}{\includegraphics{figures/0.eps}}$ and
 $\composeinline{c_1}{c_k}{d'}$ with $c_0\in\calC_i$,
 whose dimensions are enumerated by the terms $C_iE_{n-i-1}$ of the sum.
\end{proof}

For combs over a comb, $E_n=2^n$. 
For trees over a comb, $E_n$ are  the Catalan numbers.
For permutations over a comb, we have the recursion
$$
   E_0\ =\ 1\,,\quad\mbox{and for \ $n>0$,}\quad
   E_n\ =\ n! + \sum_{i=0}^{n-1}i! E_{n-i-1}\,,
$$
 which begins $1, 2, 5, 15, 54, 235,\dotsc$, and is sequence A051295 in the
 On-line Encyclopedia of Integer Sequences (OEIS)~\cite{Slo:oeis}.
 (This is the invert transform~\cite{BernSloan} of the factorial numbers.)

\begin{theorem}
When $\calD_n$ has a basis indexed by $\y_n$, then we have the  recursion
 $$
   E_0\ =\ 1\,,\quad\mbox{and for \ $n>0$,}\quad
   E_n\ =\ C_n + \sum_{i=0}^{n-1}E_i E_{n-i-1} \,.
 $$
\end{theorem}

\begin{proof}
 Again, the first term in the expression for $E_n$ is the number of basis
 elements of $\calC_n,$ since each of these trees is grafted on to
 the unit element of $\calD$.
 The sum accounts for the possible pairs of trees obtained from removing root nodes in
 $\calD$.
 In this case, each subtree from the root is another tree in $\calE$.
\end{proof}

For example, combs over a tree are enumerated by the binary
transform of the Catalan numbers~\cite{forcey2}.
Trees over a tree are enumerated by the Catalan transform of the
Catalan numbers~\cite{forcey1}.
Permutations over a tree are enumerated by the recursion
$$
   E_0\ =\ 1\,,\quad\mbox{and for \ $n>0$,}\quad
   E_n\ =\ n! + \sum_{i=0}^{n-1}E_i E_{n-i-1}\,,
$$
which begins $1, 2, 6, 22, 92, 428,\dotsc$  and is not a recognized
sequence in the OEIS~\cite{Slo:oeis}.

For $\calE = \ssym \circ \calC$, we do not have a recursion, but do have a
formula from direct inspection of the possible trees $\composeinline{c_0}{c_k}{d}$ with
$|d|=k$ (since $|\s_k| = k!$)
\[
   E_n\ =\ \sum_{k=0}^{n} \,k! \! \sum_{(\gamma_0,\dotsc,\gamma_k)} C_{\gamma_0} \dotsb C_{\gamma_k}\,,
\]
the sum over all weak compositions $\gamma = (\gamma_0,\dotsc,\gamma_k)$ of $n{-}k$ into
$k{+}1$ parts ($\gamma_i \geq0$).  
Since the number of such weak compositions is 
$\binom{(n-k)+(k+1)-1}{(k+1)-1}=\binom{n}{k}$, when $\calC=\csym$ so that $C_n=1$, this
formula becomes
\[
   E_n\ =\ \sum_{k=0}^n k! \tbinom{n}{k}\ =\ 
   \sum_{k=0}^n n!/k!\,,
\]
which is sequence A000522 in the OEIS~\cite{Slo:oeis}.	

\section{Composition of Coalgebras and Hopf Modules}\label{sec: hopf}

We give conditions ensuring that a composition of coalgebras is a
one-sided Hopf algebra, interpret these in the language of operads,
and then investigate which compositions of Section~\ref{sec: cccc examples} are one-sided Hopf
algebras.

\subsection{Module coalgebras}\label{sec: covering}

Let $\calD$ be a connected graded Hopf algebra with product $m_{\calD}$, coproduct
$\Delta_{\calD}$, and unit element $1_{\calD}$.

 A map  $f: \calE \to \calD$ of graded coalgebras is a
 \demph{connection} on $\calD$ if $\calE$ is a $\calD$--module coalgebra, $f$ is a map of
 $\calD$-module coalgebras, and $\calE$ is connected.
That is, $\calE$ is an associative (left or right) $\calD$-module whose action (denoted $\star$)
commutes with the coproducts, so that
$\Delta_{\calE}(e \star d) = \Delta_{\calE}(e) \star \Delta_{\calD}(d)$,
for  $e\in \calE$ and $d \in \calD$,
\emph{and} the coalgebra map $f$ is also a module map, so that for $e\in\calE$ and
$d\in\calD$ we have
\[
  (f\otimes f )\,\Delta_{\calE}(e)\ =\ \Delta_{\calD}\, f(e)
    \qquad\mbox{and}\qquad
  f(e\star d) \ = \ m_{\calD} \,( f(e)\otimes d)\,.
\]
We may sometimes use subscripts ($f_l$ or $f_r$) on a connection $f$ to indicate that the
action is a left- or right-module action.

\begin{theorem}\label{cover_implies_Hopf}
  If $\calE$ is a connection on $\calD$, then $\calE$ is also a
  Hopf module and a comodule algebra over $\calD$.
  It is also a one-sided Hopf algebra with one-sided unit
  $\Blue{1_{\calE}}:=f^{-1}(1_{\calD})$  and antipode.
\end{theorem}

\begin{proof}
 Suppose $\calE$ is a right $\calD$-module.
 Define the product $m_{\calE}:\calE\otimes\calE \to \calE$ via the $\calD$-action:
 $m_{\calE}:=\star\circ(1\otimes f)$.
 The one-sided unit is $1_{\calE}$.
 Then $\Delta_\calE$ is an algebra map.
 Indeed, for $e,e' \in \calE$, we have
\[
  \Delta_\calE(e\cdot e') \ =\  \Delta_\calE(e\star f(e'))
  \ =\ \Delta_\calE e \star \Delta_\calD f(e')
  \ =\  \Delta_\calE e \star (f\otimes f)(\Delta_\calE e')
  \ =\  \Delta_\calE e \cdot \Delta_\calE e'\,.
\]
As usual, $\varepsilon_\calE$ is just projection onto $\calE_0$.
The unit $1_{\calE}$ is one-sided, since
\[
  e\cdot 1_{\calE}
  \ =\ e\star f(1_{\calE})
  \ =\ e\star f(f^{-1}(1_{\calD}))
  \ =\ e\star 1_{\calD}
  \ =\ e \,,
\]
but $1_\calE \cdot e = 1_\calE \star f(e)$ is not necessarily equal to $e$.
As $\calE$ is a graded bialgebra, the antipode $S$ may be defined recursively to satisfy
$m_\calE(S\otimes \id)\Delta_\calE = \varepsilon_\calE$, 
see~\ref{thm: painted is hopf}.
(If instead $\calE$ is a left $\calD$-module, then it has a left-sided unit and
right-sided antipode.) 

Define $\rho \colon\calE\to\calE\otimes\calD$ by $\rho := (1\otimes f)\, \Delta_{\calE}$. 
This gives a coaction so that $\calE$
is a Hopf module and a comodule algebra over $\calD$.
\end{proof}

\subsection{Operads and operad modules}\label{sec: operad}

Composition of coalgebras is the same product used to
define operads internal to a symmetric monoidal category~\cite[Appendix~B]{AguMah:2010}.
A \demph{monoid} in a category with a product $\bullet$ is an object $\calD$ with a
morphism $\gamma \colon \calD\bullet\calD \to \calD$ that is associative.
An \demph{operad} is a monoid in the category of graded sets with an analog of the composition
product $\circ$ defined in Section~\ref{sec: cccc main results}.

The category of connected graded coalgebras and coalgebra maps is a symmetric monoidal
category with the composition $\circ$ of coalgebras.
A \demph{graded Hopf operad} $\calD$ is a monoid in this category.
That is, $\calD$ has associative composition maps 
$\gamma\colon\calD\circ \calD \to \calD$ obeying
\[
   \Delta_\calD\gamma(a)\ =\
    (\gamma\otimes\gamma)\left(\Delta_{\calD\circ\calD}(a)\right) \
   \hbox{ \ for all \ }a\in \calD\circ\calD\,.
\]
By Theorem \ref{thm: inthopf}, $\calD$ is a Hopf algebra; this
explains our nomenclature.

A \demph{graded Hopf operad module $\calE$} is an operad module (left or right) over $\calD$
and a graded coassociative coalgebra whose module action is compatible
with its coproduct.
Write $\mu_l:\calD\circ \calE \to \calE$ and $\mu_r:\calE\circ \calD \to\calE$ for the
left and right actions, which obey, e.g., 
\[
  \Delta_\calE\mu_r(b)\ =\ (\mu_r\otimes\mu_r)\Delta_{\calE\circ\calD} b
  \qquad \mbox{for all }b\in\calE\circ\calD\,.
\]

\begin{example}
 $\ysym$ is an operad in the category of vector spaces.
 The action of $\gamma$ on $\composeinline{F_{t_0}}{F_{t_n}}{F_t}$ grafts the indexing 
 trees $t_0,\dotsc,t_n$ onto the tree $t$ and, unlike in Example \ref{ex: painted},
 forgets which nodes of the resulting tree came from $t$.
 This is  associative in the appropriate sense.
 The same action $\gamma$ makes $\ysym$ an operad in the category of connected graded
 coalgebras, and thus a graded Hopf operad.
 Finally, operads are operad modules over themselves, so $\ysym$ is also graded
 Hopf operad module.
\end{example}

\begin{remark}
 This notion differs from that of Getzler and Jones~\cite{GetJon:1}, who
 defined a Hopf operad $\calD$ to be an operad 
 where each component $\calD_n$ is a coalgebra.
\end{remark}

\begin{theorem}\label{thm: inthopf}
 A graded Hopf operad $\calD$ is also a Hopf algebra with product
\begin{gather}\label{eq: operad to algebra}
   a\cdot b\ :=\ \gamma( b \otimes \Delta^{(n)} a)
\end{gather}
 where $b\in \calD_n$ and $\Delta^{(n)}$ is the iterated coproduct from $\calD$ to
 $\calD^{\otimes(n+1)}.$
\end{theorem}

\begin{remark} 
If we swap the roles of $a$ and $b$ on the right-hand side of \eqref{eq: operad to algebra},
we also obtain a Hopf algera, for $H^{\mathrm{op}}$ is a Hopf algebra whenever $H$ is
one. Our choice agrees with the description (Section \ref{sec: Hopf examples}) of products
in $\ysym$ and $\csym$.  
\end{remark}

Before we prove Theorem \ref{thm: inthopf}, we restate an old result in the language of operads.

\begin{proposition}\label{prop:tree and comb}
 The well-known Hopf algebra structures of $\ysym$ and $\csym$
 are induced by their structure as  graded Hopf operads.
\end{proposition}

\begin{proof}
 The operad structure on $\ysym$ is the operad of planar,
 rooted,  binary trees, where composition $\gamma$ is grafting.
 The operad structure on $\csym$ is the terminal operad, which has a single element in
 each component.
 Representing the single element of degree $n$ as
 a comb of $n$ leaves, the composition $\gamma$ becomes
 grafting and combing all branches of the result.

 We check that these compositions $\gamma$ are coalgebra maps.
 For $\calD =\ysym$, the coproduct $\Delta_{\calD\circ\calD}$ in the $F$-basis 
 is the sum over possible splittings of the composite trees. 
 Then splitting an element of $e\in \calD \circ \calD$ and grafting both
 resulting trees (via $\gamma\otimes\gamma$) yields the same result as first grafting 
 ($e\to \gamma(e)$), then splitting the resulting tree.
 When $\calD=\csym$, virtually the same analysis holds, with the proviso that graftings
 are always followed by combing all branches to the right. 

 Finally, we note that the product in $\ysym$ in terms of the $F$-basis is simply
 $a\cdot b = \gamma(b\otimes \Delta^{(|b|)}a)$. 
 The same holds for $\csym$, again with the proviso that $\gamma$ is grafting, followed by
 combing. 
\end{proof}

\begin{proof}[Proof of Theorem~\ref{thm: inthopf}]
 We have $\gamma(1\otimes1)=1$ and $\gamma(b\otimes 1^{\otimes|b|+1}) = b$ by construction,
 since $\calD$ is connected. 
 Thus $1=1_\calD$ is the unit in $\calD$.

 The image of $ \id\otimes\Delta^{(n)}$ lies in $\calD\circ\calD$.
  As $\gamma$ is a map of graded coalgebras, 
 $\Delta(a\cdot b) = \Delta a \cdot \Delta b$. 
 Indeed, for $b\in\calD$ homogeneous,
 \begin{align*}
   \Delta(a\cdot b) \ =\   \Delta (\gamma(b\otimes \Delta^{(|b|)}a))
   & =\ (\gamma \otimes \gamma)(\Delta_{\calD\circ\calD}(b\otimes \Delta^{(|b|)}a))\\
   & =\ (\gamma \otimes \gamma)(( \Delta b\otimes \Delta^{(|b|)}\Delta a) )
    \ =\ \Delta a \cdot \Delta b\,.
 \end{align*}
Associativity of the product follows, since for $b,c$ homogeneous elements of $\calD$,
we have
\begin{align}
\nonumber  a\cdot(b\cdot c)
  \ =\ a\cdot\gamma(c\otimes\Delta^{(|c|)} b )\
  &=\ \gamma\left(\gamma(c\otimes\Delta^{(|c|)}b )\otimes\Delta^{(|b|+|c|)}a\right)\\
  \label{eq: assoc}  
  &=\ \gamma\left(c\otimes 
    \gamma^{\otimes(|c|+1)}(\Delta^{(|c|)}b\otimes\Delta^{(|b|+|c|)}a) \right)\\
  \nonumber 
  &=\ \gamma\left(c \otimes(\Delta^{(|c|)}a\cdot\Delta^{(|c|)}b )\right)\\
  \label{eq: alg-map} 
   &=\ \gamma(c\otimes\Delta^{(|c|)}(a\cdot b) )
  \ =\ (a\cdot b)\cdot c\,.
\end{align}
 Here, \eqref{eq: assoc} is by the associativity of composition $\gamma$ in an operad, where we assume the isomorphism
 $\calD \circ (\calD \circ \calD) \cong (\calD \circ \calD) \circ \calD$.
 The step \eqref{eq: alg-map} follows as $\calD$ is a bialgebra ($\Delta^{(n)}$ is an
 algebra map since $\Delta=\Delta^{(1)}$ is one). 
\end{proof}

\begin{lemma}\label{lem:opmod}
  If $\calC$ is a graded coalgebra and $\calD$ is a graded Hopf operad, then
  $\calD \circ \calC$ is a (left) graded Hopf operad module and
  $\calC \circ \calD$ is a (right) graded Hopf operad module.
\end{lemma}

\begin{proof}
  We grade $\calD \circ \calC$ and $\calC \circ \calD$ by total degree.
 An operad module of vector spaces is a sequence of vector spaces
 acted upon by the operad.
 The action $\mu_l:\calD\circ(\calD \circ \calC)\to(\calD \circ \calC)$ is given by
\[
     \mu_l\left(d\otimes \compose{c_{0_0}}{c_{i_0}}{d_0}\otimes\dots\otimes
    \compose{c_{0_n}}{c_{i_n}}{d_n}\right)\
    =\ \compose{c_{0_0}}{c_{i_n}}{\gamma(d\otimes d_0\otimes\dots\otimes d_n)}\ .
\]
 Associativity of $\gamma$ implies that this  action is associative.
The action
$\mu_r:(\calC \circ \calD)\circ\calD\to(\calC \circ \calD)$ is given by
\begin{multline*}
  \qquad
   \mu_r\left(\compose{d_{0}}{d_{m}}{c}\otimes d_{0_0}\otimes\dots \otimes d_{j_m}\right)\\
   = \
   \compose{\gamma( d_0\otimes d_{0_0}\otimes\dots\otimes d_{j_0})}%
    {\gamma(d_m\otimes d_{0_m}\otimes\dots\otimes d_{j_m})}{c}\ .
\end{multline*}
 Associativity of $\gamma$ implies that this action is associative as well. We leave the
 reader to check that $\Delta\mu_l = (\mu_l\otimes \mu_l)\Delta$ and 
 $\Delta\mu_r = (\mu_r\otimes \mu_r)\Delta$.
\end{proof}

\begin{lemma}\label{lem:coal}
 A graded Hopf operad module $\calE$ over a graded Hopf operad
 $\calD$ is also a module coalgebra for the Hopf algebra $\calD$.
\end{lemma}

\begin{proof}
 Fix $e\in \calE$ and $d\in\calD$ to be homogeneous elements.
 If $\calE$ is a right operad module
 over $\calD$ then define a left action by $d\star e :=\mu_r(e \otimes \Delta^{(|e|)} d)$.
 If $\calE$ is a left operad module over
 $\calD$ then $e\star d := \mu_l(d \otimes \Delta^{(|d|)} e)$ defines a right action.

 Checking that either case defines an associative action and a module coalgebra uses the
 same reasoning as for the proof of Theorem~\ref{thm: inthopf}, with $\mu$ replacing
 $\gamma$. 
\end{proof}

\begin{theorem}\label{thm:suffcover}
  Given a coalgebra map $\lambda \colon \calC \to \calD$ from a  connected graded coalgebra $\calC$
  to a graded Hopf operad $\calD$, the maps
  $f_r = \gamma\circ(\id\circ\lambda)\colon \calD\circ\calC\to \calD$ and
  $f_l = \gamma\circ(\lambda\circ \id)\colon \calC\circ\calD\to \calD$ give
  connections on $\calD$.
\end{theorem}

\begin{proof}
  By Theorem~\ref{thm: inthopf} and Lemmas~\ref{lem:opmod}
  and~\ref{lem:coal}, $\calD\circ\calC$ and $\calC\circ\calD$ are connected graded module
  coalgebras over $\calD$.
  We need only show that the maps $f_r$
  and $f_l$ are coalgebra maps and module maps. In terms of decomposable tensors, the maps
  take the form, 
\begin{gather*}
   f_r\left(\compose{c_0}{c_n}{d}\right)\ :=\
  \gamma\left(\compose{\lambda(c_0)}{\lambda(c_n)}{d}\right)
\quad\hbox{and}\quad
 f_l\left(\compose{d_0}{d_n}{c}\right)\ := \
 \gamma\left(\compose{d_0}{d_n}{\lambda(c)}\right).
\end{gather*}

These are coalgebra maps since both $\lambda$ and $\gamma$ are coalgebra maps.
The associativity of $\gamma$ implies that $f_r$ and
 $f_l$ are maps of right and left $\calD$-modules, respectively.
\end{proof}

\subsection{Examples of module coalgebra connections}\label{sec: Hopf op examples}

Eight of the nine compositions of coalgebras from Section \ref{sec: cccc examples}
are connections on one or both of the factors $\calC$ and $\calD$.

\begin{theorem}\label{thm:eightex}
  For $\calC\in\{\ssym,\ysym,\csym\}$, the coalgebra compositions $\calC\circ\csym$ and
  $\csym\circ\calC$ are connections on $\csym$.
  For $\calC \in \{\ssym,\ysym, \csym\}$, the coalgebra compositions $\calC\circ\ysym$ and
  $\ysym\circ\calC$ are connections on $\ysym$.
\end{theorem}

\begin{proof}
 By Theorem~\ref{thm:suffcover} and Proposition~\ref{prop:tree and comb},
 we need only show the existence of coalgebra
 maps from $\calC$ to $\calD$, for $\calC \in \{\ssym,\ysym, \csym\}$ and
 $\calD\in\{\ysym,\csym\}$. 

 For $\calD = \csym$, the maps $\bkappa\btau$, $\bkappa$, and $\id$ are all
 coalgebra maps to $\csym$ (Proposition \ref{thm: Hopf examples}).
 For $\calD = \ysym$, the maps $\btau$ and $\id$ are coalgebra maps to $\ysym$.
 Lastly, combs are binary trees, and the induced inclusion map
 $\csym \hookrightarrow \ysym$ is a coalgebra map.
\end{proof}

Note that in particular, $\ysym\circ\csym$ is a connection on both $\csym$ and $\ysym$.
This yields two distinct one-sided Hopf algebra structures on $\ysym\circ\csym$.
Likewise, $\ysym\circ\ysym$ is a connection on $\ysym$ in two distinct ways (again leading
to two distinct one-sided Hopf structures).
In the remaining sections, we discuss three of the compositions of
Section~\ref{sec: cccc examples} which have appeared previously.


\section{A Hopf Algebra of Painted Trees}\label{sec: painted}

Our motivating example is the self-composition
$\psym := \ysym\circ \ysym$.
Elements of the fundamental basis of $\psym$ are
$F_p = \composeinline{c_0}{c_{|d|}}{d}$, where $c_0,\dotsc,c_{|d|}$ and $d$ are
elements of the fundamental basis of $\ysym$.
The indexing trees of $c_1,\dotsc,c_{|d|}$ and $d$ may be combined to form painted
trees as in Example~\ref{ex: painted}.
We describe the topological origin of painted trees and their relation to the 
multiplihedron, and we relate the Hopf structures of $\psym$ to the Hopf structures 
of $\msym$ developed in~\cite{FLS:2010}.

\subsection{Painted binary trees in topology.}
 A \demph{painted binary tree} is a planar binary tree $t$,
together with a (possibly empty) upper order ideal
of its node poset.
We indicate this ideal  by painting on top of a
representation of $t$. For clarity, we stop our painting in the
middle of edges.
Here are a few simple examples,
 \begin{equation}\label{Eq:PT}
  \raisebox{-15pt}{%
    \includegraphics{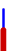}\, {,}   \qquad
    \includegraphics{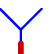}\!\! {,} \qquad
    \includegraphics{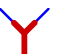} \!\! {,} \qquad
    \includegraphics{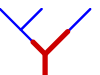}  \!\! {,} \qquad
    \includegraphics{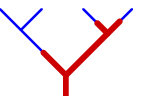}\!\! {,} \qquad
    \includegraphics{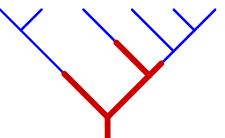} \! {.} }
 \end{equation}

An \demph{$A_n$-space} is a topological $H$-space with a weakly
associative multiplication of points~\cite{Sta:1963}.
(Products are represented by planar binary trees as these distinguish between possible
choices of associations.)
Maps between $A_{n}$-spaces preserve the multiplicative
structure only up to homotopy.
Stasheff~\cite{Sta:1963} described these maps combinatorially using
cell complexes called multiplihedra, while  Boardman and Vogt~\cite{BoaVog:1973},  used
spaces of painted  trees.
Both the spaces of trees and the cell complexes are
homeomorphic to convex polytope realizations of the multiplihedra as
shown in \cite{forcey1}.

If $f\colon ( X,\Blue{\Placeholder}) \to (Y,\Red{\ast})$ is a map of  $A_n$-spaces, then the
different ways to multiply and map $n$ points of
$X$ are represented by painted trees.
Unpainted nodes are multiplications in $X$, painted nodes are
multiplications in $Y$, and the beginning of the painting indicates that $f$ is applied to a
given point in $X$.
 See Figure \ref{fig: maps to painted}.
\begin{figure}[htb]
\[
    f({\blue a})\,{\red\bm\ast}\,\bigl(f({\blue b\,\Placeholder\,c})\,
     {\red\bm\ast}\,f({\blue d})\bigr)
    \ \longleftrightarrow \
    {\raisebox{-.5\height}{\includegraphics{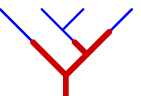}}}
\]
\caption{$A_n$-maps between $H$-spaces correspond to painted binary trees.}
\label{fig: maps to painted}
\end{figure}

Figure~\ref{fig:_multi_paint} shows the three-dimensional multiplihedron with its
vertices labeled by painted trees having three internal nodes. This
picture of the multiplihedron also shows that the
vertices are the elements of a lattice whose Hasse diagram is the one-skeleton
of the polytope in the view shown. See \cite{FLS:2010} for an explicit description of the
covering relations in terms of \emph{bi-leveled trees}. 

\subsection{Painted trees as bi-leveled trees} 
 
A \demph{bi-leveled tree} is a planar binary tree $t$ together with an order ideal $\setT$
of its node poset which contains the leftmost node, but none of its children.
We display bi-leveled trees corresponding to the painted trees of~\eqref{Eq:PT},
circling the nodes in $\setT$.
\[
   \includegraphics{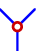}\, {,} \qquad
   \includegraphics{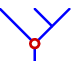}\!\! {,}\qquad
   \includegraphics{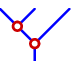}\!\! {,}\qquad
   \includegraphics{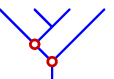}\!\! {,}\qquad
   \includegraphics{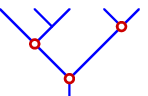}\!\! {,}\qquad
   \includegraphics{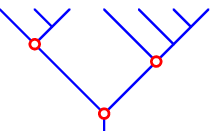}\! {.}
\]
Bi-leveled trees having $n{+}1$ internal nodes are in bijection with
painted trees having $n$ internal nodes, the bijection being given by pruning:
Remove the leftmost branch and node from a bi-leveled tree to get a tree whose order ideal
is the order ideal of the bi-leveled tree, minus
the leftmost node.
For an illustration of this and the inverse mapping, see Figure~\ref{fig: painted to bi-leveled}.
\begin{figure}[htb]
\[
   \includegraphics{figures/p2413.d.eps}
\ \xleftrightarrow{\raisebox{-12pt}{\ 
    \includegraphics{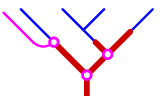}}\ \ }
\
  \includegraphics{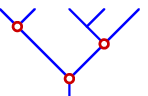} 
\]
\caption{Painted trees correspond to bi-leveled trees.}\label{fig: painted to bi-leveled}
\end{figure}

Let $\m_n$ be the set of bi-leveled trees with $n$ internal nodes.
In \cite{FLS:2010} we developed several algebraic structures on the graded vector space
$\msym$ with basis $F_b$ indexed by bi-leveled trees $b$, graded by the number of internal
nodes of $b$.
We also placed a $\ysym$--Hopf module structure on $\msym_+$, the positively graded part of $\msym$. We revisit this structure in Section \ref{sec: add structs}.

\subsection{The coalgebra of painted trees.}\label{sec: coalgebra psym}

Let $\p_n$ be the poset of painted trees on $n$ internal nodes, with
partial order inherited from the identification with bi-leveled trees $\m_{n+1}$.
We show $\p_3$ in Figure~\ref{fig:_multi_paint}.
\begin{figure}[htb]

\[
  \begin{picture}(260,238)(2,2)
   \put(2,5){\includegraphics[height=240pt]{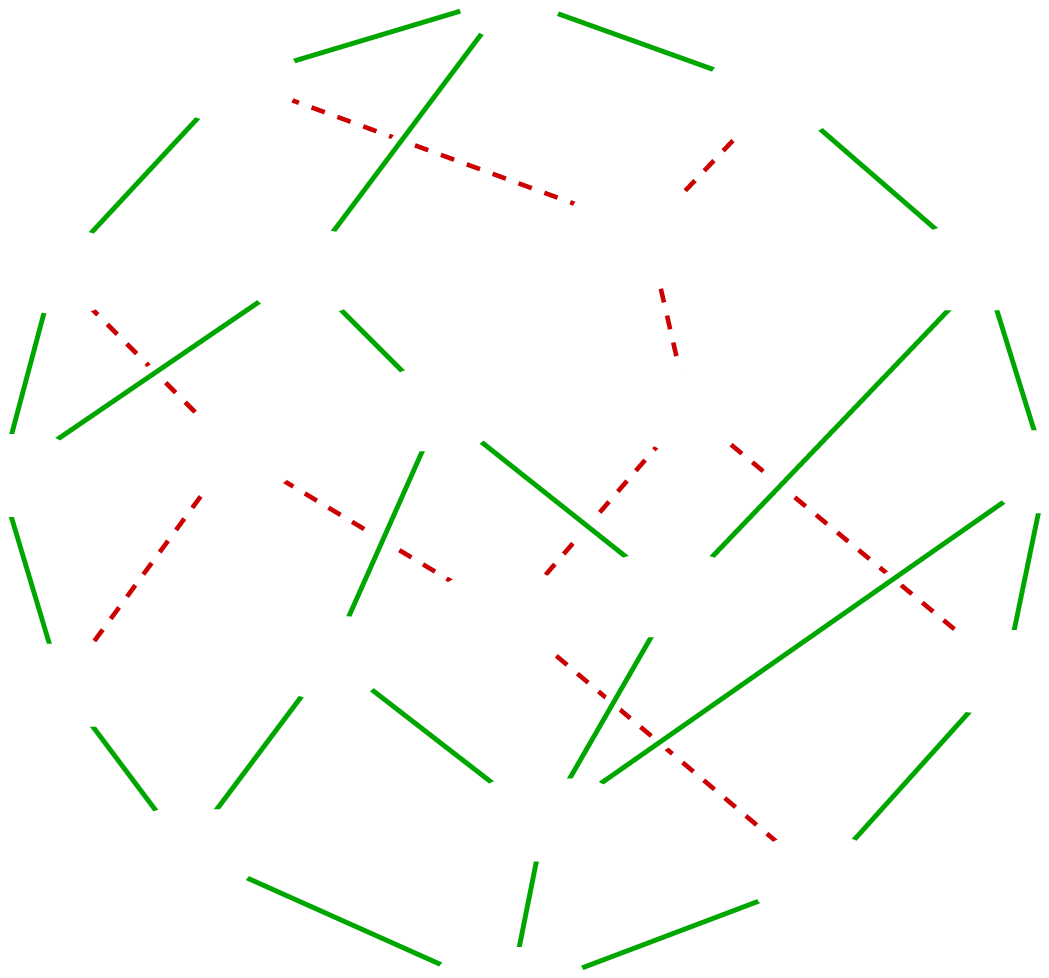}}

   \put(119,224){\includegraphics{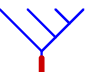}}

   \put(56,204){\includegraphics{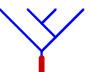}}
   \put(182,203){\includegraphics{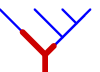}}

   \put(19,163){\includegraphics{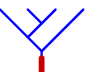}}
   \put(74,163){\includegraphics{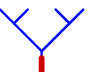}}
   \put(149,171){\includegraphics{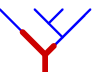}}
   \put(225,164){\includegraphics{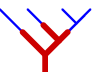}}

   \put( 4,116.5){\includegraphics{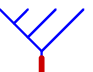}}
   \put(58,121){\includegraphics{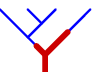}}
   \put(105,130.7){\includegraphics{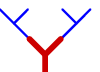}}
   \put(160.8,130.5){\includegraphics{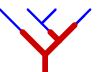}}
   \put(235,118){\includegraphics{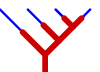}}

   \put(21, 70){\includegraphics{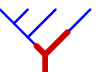}}
   \put(80.3, 74){\includegraphics{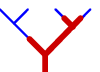}}
   \put(121, 84){\includegraphics{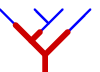}}
   \put(155.5, 90){\includegraphics{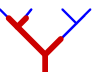}}
   \put(228, 72){\includegraphics{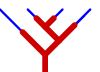}}

   \put(49, 32){\includegraphics{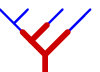}}
   \put(128.5, 38.5){\includegraphics{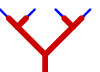}}
   \put(185, 24){\includegraphics{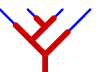}}

   \put(121, 0){\includegraphics{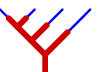}}
   \end{picture}
\]
\caption{The one-skeleton of the three-dimemsional multiplihedron, $\m_4$.}
\label{fig:_multi_paint}
\end{figure}
 We refer to~\cite{FLS:2010} for a description of the order on $\m_{n+1}$.
(Note that the map from $\p_\bb$ to $\m_\bb$ actually lands in
$\m_+$, which consists of the trees in $\m_\bb$ with one or more nodes.)

We reproduce the compositional coproduct defined in
Section~\ref{sec: cccc main results}.

\begin{definition}[Coproduct on $\psym$]\label{def: painted is cccc}
 Given a painted tree $p$, define the coproduct in the fundamental
 basis $\bigl\{ F_p \mid p \in \p_\bb \bigr\}$ by
 \begin{gather*}
        \Delta(F_p)\ =\ \sum_{p \psplit (p_0,p_1)} F_{p_0} \otimes F_{p_1} ,
 \end{gather*}
 where the painting in $p$ is preserved in the splitting $p\psplit(p_0,p_1)$.
\end{definition}

The counit $\varepsilon$ is projection onto $\psym_0$, which is spanned by
$F_{\includegraphics{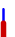}}$. 

Theorem~\ref{thm: cc cofree} describes the primitive elements of
$\psym = \ysym \circ \ysym$ in terms of the primitive elements of $\ysym$.
We recall the description of primitive elements of $\ysym$ as given
in~\cite{AguSot:2006}.
The set of trees $\y_n$ forms a poset whose covering relation is obtained by
moving a child node of a given node from the left to the right branch above the
given node.
Thus
 \[
  \raisebox{-7pt}{\includegraphics{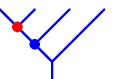}}
      {\color{red} \ \longrightarrow\  }
  \raisebox{-7pt}{\includegraphics{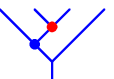}}
    {\color{blue}   \ \longrightarrow\  }
  \raisebox{-7pt}{\includegraphics{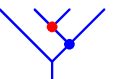}}
      {\color{red} \ \longrightarrow\   }
  \raisebox{-7pt}{\includegraphics{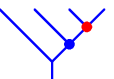}}
 \]
is an increasing chain in $\y_3$ (the moving vertices are marked with dots).

Let $\mu$ be the M\"obius function of $\y_n$ which is defined by
$\mu(t,s)=0$ unless $t\leq s$,
\[
   \mu(t,t)\ =\ 1\,, \qquad\mbox{and}\qquad
   \mu(t,r)\ =\ -\sum_{t\leq s< r} \mu(t,s)\,.
\]
We define a new basis for $\ysym$ using the M\"obius function.
For $t\in\y_n$, set
\[
   M_t\ :=\ \sum_{t\leq s} \mu(t,s) F_s\,.
\]
Then the coproduct for $\ysym$ with respect to this $M$-basis is still given by splitting
of trees, but only at leaves emanating directly from the right branch above the root:
\[
  \Delta( M_{\;\includegraphics{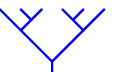}})\ =\
  1 \otimes M_{\;\includegraphics{figures/42513.eps}} \ +\
  M_{\;\includegraphics{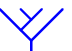}} \otimes M_{\;\includegraphics{figures/12.eps}}\ +\
  M_{\;\includegraphics{figures/42513.eps}}\otimes 1\,.
\]
A tree $t\in\y_n$ is \demph{progressive} if it has no branching along the right branch
above the root node.
A consequence of the description of the coproduct in this $M$-basis
is Corollary 5.3 of~\cite{AguSot:2006} that the set
$\{M_t \mid t \hbox{ is progressive}\}$ is a linear basis for the space of primitive
elements in $\ysym$.

Thus according to Theorem~\ref{thm: cc cofree} the cogenerating
primitives in $\psym$ are of two types:
\[
  \lrcompose{1}{c_1}{c_{n-1}}{1}{M_t} \qquad\hbox{ and }\qquad \zcompose{\,M_t\,}{1} \,,
\]
where $t$ is a progressive tree.

Here are some examples of the first type, 
 \begin{align*}
  M_{\;\includegraphics{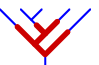}} \ &:= \
  \zcompose{1\treeseparator F_{\;\includegraphics{figures/1.eps}}\treeseparator 1
      \treeseparator 1}{M_{\;\includegraphics{figures/213.eps}}}
   \ = \
    F_{\;\includegraphics{figures/p24135.eps}}\ - \
    F_{\;\includegraphics{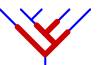}}\ ,\\
M_{\;\includegraphics{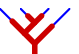}}\ &:= \
\zcompose{1\treeseparator 1\treeseparator 1\treeseparator
1}{M_{\;\includegraphics{figures/213.eps}}}
   \ = \
    F_{\;\includegraphics{figures/p1324.eps}} -
    F_{\;\includegraphics{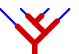}}  \,,  \rule{0pt}{20pt}
 \\
 M_{\;\includegraphics{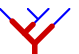}}\ &:=\
\zcompose{1\treeseparator F_{\;\includegraphics{figures/1.eps}}\treeseparator 1}%
       {M_{\;\includegraphics{figures/12.eps}}}
   \ = \
    F_{\;\includegraphics{figures/p2314.eps}} -
    F_{\;\includegraphics{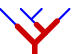}}\ ,   \rule{0pt}{20pt}
\intertext{and one of the second type,}
M_{\;\includegraphics{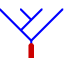}}\ &:= \
 \zcompose{M_{\;\includegraphics{figures/213.eps}}}{1}
   \ = \
    F_{\;\includegraphics{figures/p4213.eps}} -
    F_{\;\includegraphics{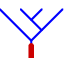}}  \,.  \rule{0pt}{20pt}
 \end{align*}

The primitives can be described in terms of M\"obius inversion on certain subintervals of
the multiplihedra lattice.
For the first type, the subintervals are those with a fixed unpainted forest of the form
$(\includegraphics{figures/0.eps}, t, \dotsc, s,\includegraphics{figures/0.eps})$.
For the second type, the subinterval consists of those trees whose painted part is trivial,
$\includegraphics{figures/p0.eps}$.
Each subinterval of the first type is isomorphic to $\y_m$ for some $m\leq n$, and the second
subinterval is isomorphic to $\y_{n}$.
Figure~\ref{fig:_multi_sub} shows the multiplihedron lattice $\p_3$, with these
subintervals highlighted.
\begin{figure}[htb]
\[
  \begin{picture}(260,238)(2,2)
   \put(2,5){\includegraphics[height=240pt]{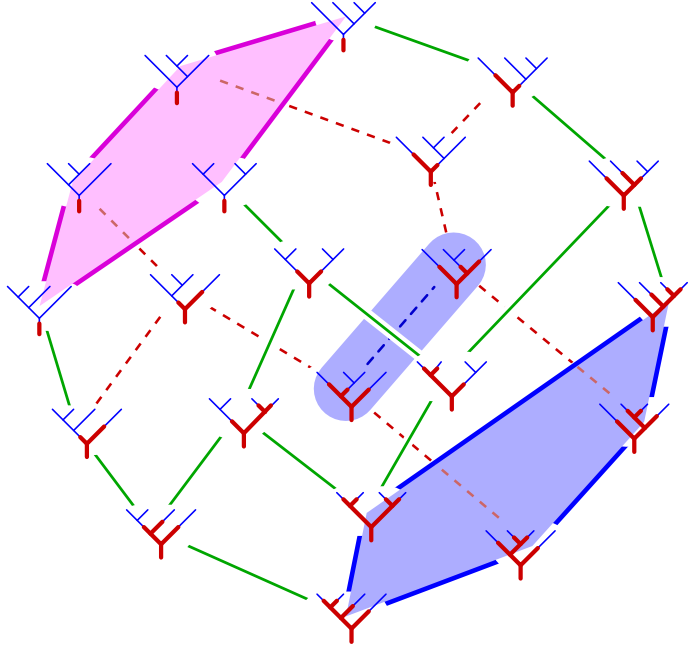}}

   \put(119,224){\includegraphics{figures/p4321.s.eps}}

   \put(56,204){\includegraphics{figures/p4312.s.eps}}
   \put(182,203){\includegraphics{figures/p3421.s.eps}}

   \put(19,163){\includegraphics{figures/p4213.s.eps}}
   \put(74,163){\includegraphics{figures/p4231.s.eps}}
   \put(151,173){\includegraphics{figures/p3412.s.eps}}
   \put(224,164){\includegraphics{figures/p2431.s.eps}}

   \put( 4,116.5){\includegraphics{figures/p4123.s.eps}}
   \put(58,121){\includegraphics{figures/p3214.s.eps}}
   \put(105,130.7){\includegraphics{figures/p3241.s.eps}}
   \put(160.8,130.5){\includegraphics{figures/p2413.s.eps}}
   \put(235,118){\includegraphics{figures/p1432.s.eps}}

   \put(21, 70){\includegraphics{figures/p3124.s.eps}}
   \put(80.3, 74){\includegraphics{figures/p2143.s.eps}}
   \put(121, 84){\includegraphics{figures/p2314.s.eps}}
   \put(159, 88){\includegraphics{figures/p2341.s.eps}}
   \put(228, 72){\includegraphics{figures/p1423.s.eps}}

   \put(49, 32){\includegraphics{figures/p2134.s.eps}}
   \put(128.5, 38.5){\includegraphics{figures/p1243.s.eps}}
   \put(185, 24){\includegraphics{figures/p1324.s.eps}}

   \put(121, 0){\includegraphics{figures/p1234.s.eps}}

   \end{picture}
\]
\caption{The multiplihedron lattice $\m_4$ showing the three
subintervals that yield primitives via M\"obius inversion.}
\label{fig:_multi_sub}
\end{figure}

\subsection{Hopf structures on painted trees.}\label{sec: add structs}

As determined in the proof of Theorem \ref{thm:eightex}, the identity map $\id:\ysym\to\ysym$
yields a connection $f_r \colon \psym \to \ysym$. In particular
(Theorem~\ref{cover_implies_Hopf}), $\psym$ is a one-sided Hopf algebra, a $\ysym$--Hopf
module, and a $\ysym$--comodule algebra. 
We discuss these structures, and relate them to structures placed on $\msym$
in~\cite{FLS:2010}. 

Let $p,q$ be painted trees with $|q|=n$.
In terms of the $F$-basis, $f_r$ simply forgets the painting level, e.g., 
$f_r(F_{\;\includegraphics{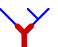}}) =F_{\;\includegraphics{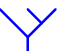}}$. 
Thus Theorem~\ref{cover_implies_Hopf} describes the product $F_p\cdot F_q$ in $\psym$
as
 \[
    F_p\cdot F_q\  =\  \sum_{p\psplit (p_0,p_1,\dots,p_n)} F_{(p_0,p_1,\dots,p_n)/q^+} \,,
 \]
where the painting in $p$ is preserved in the splitting
$(p_0,p_1,\dots,p_n)$, and $q^+$ signifies that $q$ is painted
completely before grafting.
 Here is an example of the product,
\[
     F_{\;\includegraphics{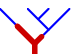}}\cdot
     F_{\;\includegraphics{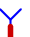}}
     \ =\
     F_{\;\includegraphics{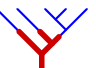}}\ +\
     F_{\;\includegraphics{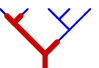}}\ +\
     F_{\;\includegraphics{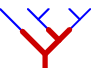}}\ +\
     F_{\;\includegraphics{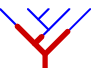}}\,.
\]

The painted tree \includegraphics{figures/p0.eps} with $0$ nodes is only
a right unit: for all $q\in\p_\bb$,
\[
    F_{\includegraphics{figures/p0.eps}} \cdot F_q\ =\
   F_{q^+} \qquad\hbox{and}\qquad F_q\cdot F_{\includegraphics{figures/p0.eps}}\ =\ F_q \,.
\]

Although the antipode is guaranteed to exist, we include a
proof for purpose of exposition.

\begin{theorem}\label{thm: painted is hopf}
 There are unit and antipode maps $\eta\colon \K \to \psym$ and
 $S \colon \psym \to \psym$ making $\psym$ a one-sided Hopf algebra.
\end{theorem}

\begin{proof}
 We just observed that $\eta\colon1\mapsto F_{\includegraphics{figures/p0.eps}}$
 is a right unit for $\psym$. 
 We verify that a \emph{left antipode} exists.
 That is, there exists a map $S\colon\psym\to\psym$ such that
 $S(F_{\includegraphics{figures/p0.eps}})=F_{\includegraphics{figures/p0.eps}}$, and for $p\in\p_+$, we
 have 
 \begin{equation}\label{Eq:antipode_def}
  \sum_{p \psplit (p_0,p_1)} S(F_{p_0})\cdot F_{p_1} \ =\ 0\,.
 \end{equation}
 Since $\psym$ is graded, and $|p|=|p_0|+|p_1|$ whenever
 $p \psplit (p_0,p_1)$, we may recursively construct $S$ using induction on $|p|$.
 First, set
 $S(F_{\includegraphics{figures/p0.eps}})=F_{\includegraphics{figures/p0.eps}}$.
 Then, given any painted tree $p$, the only term involving $S(F_q)$
 in~\eqref{Eq:antipode_def} with $|q|=|p|$ is
 $S(F_p)\cdot F_{\includegraphics{figures/p0.eps}}=S(F_p)$,
 and so we may solve~\eqref{Eq:antipode_def} for $S(F_p)$ to obtain
\[
     S(F_p)\ :=\  -\!\sum_{\substack{p \psplit (p_0,p_1)\\
         \rule{0pt}{1.4ex}|p_0|,|p_1|>0}} S(F_{p_0})\cdot F_{p_1} \ -\
   S(F_{\includegraphics{figures/p0.eps}})\cdot F_{p}\,,
 \]
 expressing $S(F_p)$ in terms of previously defined values $S(F_q)$.
\end{proof}

For example, 
\begin{align*}
   S(F_{\;\includegraphics{figures/p21.eps}})\ 
   &=\ -S(F_{\includegraphics{figures/p0.eps}})\cdot
    F_{\;\includegraphics{figures/p21.eps}}\ =\ 
    -F_{\;\includegraphics{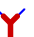}}\,,\qquad\mbox{and}\\
  S(F_{\;\includegraphics{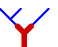}})\ 
  &=\ -S(F_{\;\includegraphics{figures/p21.eps}})\cdot F_{\;\includegraphics{figures/p12.eps}}
   \ -\ S(F_{\;\includegraphics{figures/p0.eps}})\cdot
  F_{\;\includegraphics{figures/p213.eps}}\ =\
    F_{\;\includegraphics{figures/p12.eps}}\cdot F_{\;\includegraphics{figures/p12.eps}}
   \ -\ F_{\;\includegraphics{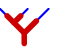}}\\
  &=\ F_{\;\includegraphics{figures/p123.eps}}   
  \ +\   F_{\;\includegraphics{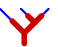}}   
  \ -\  F_{\;\includegraphics{figures/p123.eps}}\ =\ 
   F_{\;\includegraphics{figures/p132.eps}} \,.
\end{align*}

\begin{remark}\label{rem: one-sided}
 One may be tempted to artificially adjoin a true unit $e$ to $\psym$, but this only pushes
 the problem to the antipode map: $S(F_{\includegraphics{figures/p0.eps}})$ 
 cannot be defined if $\eta(1) = e$.
\end{remark}

The $\ysym$-Hopf module structure on $\psym$ of
Theorem~\ref{cover_implies_Hopf} has  coaction,
 \[
    \rho(F_p)\  =\  \sum_{p \psplit (p_0,p_1)} F_{p_0} \otimes F_{f(p_1)}\,,
 \]
where the painting in $p$ is preserved in the first half of the
splitting $(p_0,p_1)$, and forgotten in the second half.

Under the bijection between $\p_\bb$ and  $\m_+$ that
grows an extra node as in Figure \ref{fig: painted to bi-leveled},
the splittings and graftings on $\psym$ map to the restricted
splittings $\rsplit$ and graftings defined in \cite[Section~4.1]{FLS:2010}. 
Moreover, we can split and graft before or after the
bijection to achieve the same results. These facts allow the following corollary.

\begin{corollary}\label{thm: msym+ is Hopf}
The $\ysym$ action and coaction defined in \cite[Section~4.1]{FLS:2010}
make $\msym_+$ into a Hopf module isomorphic to the Hopf module
$\psym$. \hfill \qed
\end{corollary}

The \demph{coinvariants} of a Hopf module
$\rho\colon\calE\to\calE\otimes\calD$ are elements $e\in\calE$ such that  
$\rho(e) = e \otimes 1$. 
The coinvariants for the action of Corollary~\ref{thm: msym+ is Hopf} were described
explicitly in \cite[Corollary~4.5]{FLS:2010}.  
In contrast to the discussion in Section \ref{sec: coalgebra psym},  M\"obius inversion in
the entire lattice $\m_\bb$ helps to find the coinvariants.


\section{A Hopf Algebra of Weighted Trees}\label{sec: weighted}

The composition of coalgebras $\ysym\circ\csym$
has fundamental basis indexed by forests of combs attached to binary trees,
which we will call weighted trees.
By the first statement of Theorem~\ref{thm:eightex}, it has a connection on
$\csym$ that gives it the structure of a one-sided Hopf algebra.
We examine this Hopf algebra in more detail.

\subsection{Weighted trees in topology}

In a forest of combs attached to a binary tree,
the combs may be replaced by corollae  or by a positive
\demph{weight} counting
the number of leaves in the comb.
These all give \demph{weighted trees}.
 \begin{equation}\label{Eq:composite_trees}
  \raisebox{-17pt}{${\displaystyle
   \includegraphics{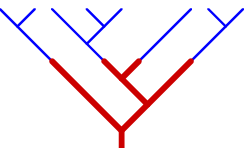} \quad\raisebox{17pt}{=}\quad
   \includegraphics{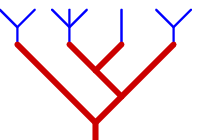} \quad\raisebox{17pt}{=}\quad
   \begin{picture}(50,37.5)(1,-3)
    \put(3,-3){\includegraphics{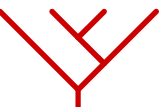}}
    \put( 0,28){$2$}    \put(15,28){$3$}
    \put(30,28){$1$}    \put(45,28){$2$}
   \end{picture}}$}
 \end{equation}
Let $\ck_n$ denote the weighted trees with weights summing to $n{+}1$. These index the
vertices of the $n$-dimensional \demph{composihedron},
$\ck(n{+}1)$~\cite{forcey2}. This sequence of polytopes parameterizes
homotopy maps between strictly associative and homotopy associative
$H$-spaces. Figure~\ref{F:comp} gives a picture of the composihedron
$\ck_3$.
\begin{figure}[hbt]
\[
  \begin{picture}(210,192)(0,2)
   \put(  8,   5){\includegraphics{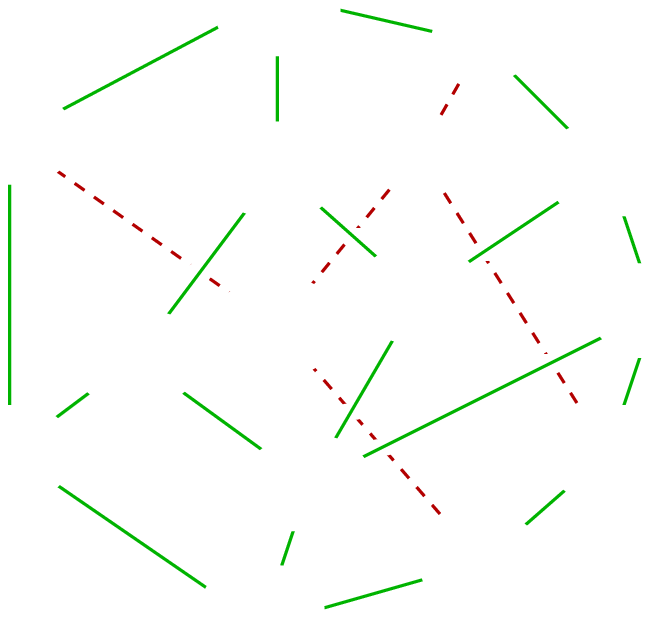}}
   \put( 77, 177){\includegraphics{figures/p4321.s.eps}}

   \put(135, 164){\includegraphics{figures/p3421.s.eps}}

   \put(  3, 136){\includegraphics{figures/p3214.s.eps}}
   \put( 75, 130){\includegraphics{figures/p3241.s.eps}}
   \put(117, 133){\includegraphics{figures/p2413.s.eps}}
   \put(170, 129){\includegraphics{figures/p2431.s.eps}}

   \put( 34,  75){\includegraphics{figures/p2143.s.eps}}
   \put( 77,  84){\includegraphics{figures/p2314.s.eps}}
   \put(116,  92){\includegraphics{figures/p2341.s.eps}}
   \put(183,  88){\includegraphics{figures/p1432.s.eps}}

   \put(  3, 47){\includegraphics{figures/p2134.s.eps}}
   \put( 83, 37){\includegraphics{figures/p1243.s.eps}}
   \put(135, 15){\includegraphics{figures/p1324.s.eps}}
   \put(168, 48){\includegraphics{figures/p1423.s.eps}}

   \put(75, 0){\includegraphics{figures/p1234.s.eps}}

  \end{picture}
\]
\caption{The one-skeleton of the three-dimensional composihedron.
} \label{F:comp}
\end{figure}
For small values of $n$, the composihedra $\ck(n)$ also appear as the
commuting diagrams in enriched bicategories~\cite{forcey2}.
These diagrams appear in the definition of pseudomonoids~\cite[Appendix~C]{AguMah:2010}.


\subsection{A Hopf algebra of weighted trees}

We describe the key definitions of Section~\ref{sec: cccc main results} and Section~\ref{sec: hopf}
for $\Blue{\cksym}:=\ysym\circ\csym$.
In the fundamental basis $\bigl\{ F_p \mid p \in \ck_\bb \bigr\}$ of $\cksym$,
the coproduct is
 \begin{gather*}
        \Delta(F_p)\ =\ \sum_{p \psplit (p_0,p_1)} F_{p_0} \otimes F_{p_1}\,,
 \end{gather*}
where the painting in $p\in\ck_\bb$ is preserved in the splitting $p\psplit(p_0,p_1)$. 
The counit $\varepsilon$ is projection onto $\cksym_0$, which is spanned by $F_{\includegraphics{figures/p0.eps}}.$
Here is an example in terms of weighted trees,
\newcommand{\nsp}{\hspace{-0.3pt}}
\[
   \Delta(F_{\CTIT{2}{\nsp 1}{2}})\ =\
    F_{\CTO{1}}\otimes F_{\CTIT{2}{\nsp 1}{2}} \ +\
    F_{\CTO{2}}\otimes F_{\CTIT{1}{1}{2}} \ +\
    F_{\CTI{2}{1}}\otimes F_{\CTI{1}{2}} \ +\
    F_{\CTIT{2}{1}{1}}\otimes F_{\CTO{2}}  \ +\
    F_{\CTIT{2}{\nsp 1}{2}}\otimes F_{\CTO{1}} \,.
\]

The primitive elements of $\cksym=\ysym\circ\csym$
have the form
\[
   F_{\;\CTO{2}}\ =\ \zcompose{\,F_{\;\includegraphics{figures/1.eps}}\,}{1}
   \qquad\mbox{ and }\qquad
    \lrcompose{1}{c_1}{c_{n-1}}{1}{M_t}\,,
\]
where $t$ is a progressive tree with $n$ nodes and $c_1,\dotsc,c_{n-1}$ are any elements
of $\csym$.
In terms of weighted trees, the indices of the second type are 
weighted progressive trees with weights of 1 on their leftmost and rightmost leaves.

Let $f_l\colon\cksym \to \csym$ be the connection given by Theorem \ref{thm:suffcover} (built from the coalgebra map $\bkappa$).
Then Theorem~\ref{cover_implies_Hopf} gives the product
\[
   F_p \cdot F_q\ :=\ f_l(F_p)\star F_q\,,\qquad
   \mbox{where }p,q\in\ck_\bb\,.
\]
In terms of the $F$-basis, $f_l$ acts on indices, sending a weighted tree $p$ to the unique
comb $f_l(p)$ with the same number of nodes as $p$.
The action $\star$ in the $F$-basis is given as follows: 
split $f_l(p)$ in all ways to make a forest of $|q|{+}1$ combs; 
graft each splitting onto the leaves of the forest of combs in $q$; 
comb the resulting forest of trees to get a new forest of combs. 
We illustrate one term in the product. 
Suppose that
$p=\CTI{2}{1}=\includegraphics{figures/p213.eps}$
and $q = \CTIT{1}{2}{1}=\includegraphics{figures/p2314.eps}$.
Then $f_l(p) =\includegraphics{figures/21.eps}$ 
and one way to split $f_l(p)$ gives the forest
$(\includegraphics{figures/0.eps}\;,\;
  \includegraphics{figures/1.eps}\;,\;
  \includegraphics{figures/0.eps}\;,\;
  \includegraphics{figures/1.eps}\;)$.
   Grafting this onto $q$ gives
   \raisebox{-3pt}{\includegraphics{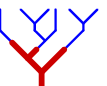}}\!,
   which after combing the forest yields the term
   $\raisebox{-3pt}{\includegraphics{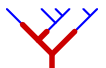}}
   =\CTIT{1}{3}{2}$ in the product $p\cdot q$.
   Doing this for the other nine splittings of $f_l(p)$ gives 
\[
    F_{\CTI{2}{1}}\cdot F_{\CTIT{1}{2}{1}} \ =\
	    F_{\CTIT{3}{2}{1}}\ +\   3  F_{\CTIT{1}{4}{1}}\ +\
	    F_{\CTIT{1}{2}{3}}\ +\   2  F_{\CTIT{2}{3}{1}}\ +\
	    F_{\CTIT{2}{2}{2}}\ +\    2 F_{\CTIT{1}{3}{2}}\,.
\]


\section{Composition trees and the Hopf algebra of simplices}\label{sec: simplices}

The simplest composition of Section~\ref{sec: cccc examples} is $\csym\circ\csym$.
As shown in Section~\ref{S:enumeration}, the graded component of total degree $n$ has
dimension $2^n$, indexed by trees with $n$ interior nodes obtained by grafting a forest of combs to
the leaves of a comb (which is painted).
Analogous to~\eqref{Eq:composite_trees}, these are weighted combs. 
As these are in bijection with compositions of $n{+}1$, we refer to them as \demph{composition
  trees}. 
\[
   \includegraphics{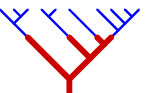}\
   \raisebox{10pt}{$\ =\ $}\
   \begin{picture}(28,24)
    \put(2.5,0){\includegraphics{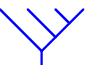}}
    \put( 0,17.5){\small 3} \put( 8,17.5){\small 2}
    \put(15.5,17.5){\small 1} \put(24,17.5){\small 4}
   \end{picture}\
   \raisebox{10pt}{$\ =  \  (3,2,1,4)$\,.}
\]

\subsection{Hopf algebra structures on composition trees}

The coproduct may again be described via splitting.
Since the composition tree $(1,3)$
has the four splittings
 \begin{equation}\label{Eq:comp_split}
   \raisebox{-3pt}{\includegraphics{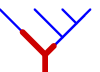}}\
   \raisebox{2pt}{$\xrightarrow{\ \curlyvee\ }$}\
   \Bigl(\;\raisebox{-3pt}{\includegraphics{figures/p0.d.eps}}\,,
         \raisebox{-3pt}{\includegraphics{figures/p3421.d.eps}}\Bigr)
   \,,\quad
   \Bigl(\raisebox{-3pt}{\includegraphics{figures/p12.d.eps}}\,,
         \raisebox{-3pt}{\includegraphics{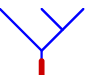}}\;\Bigr)
    \,,\quad
   \Bigl(\;\raisebox{-3pt}{\includegraphics{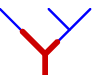}}\,,
         \raisebox{-3pt}{\includegraphics{figures/p21.d.eps}}\Bigr)
    \,,\quad
   \Bigl(\;\raisebox{-3pt}{\includegraphics{figures/p3421.d.eps}},\;
         \raisebox{-3pt}{\includegraphics{figures/p0.d.eps}}\;\Bigr)\,,
 \end{equation}
we have $\Delta(F_{1,3}) = F_{1}\otimes F_{1,3}+F_{1,1}\otimes F_{3}+
   F_{1,2}\otimes F_{2}+ F_{1,3}\otimes F_{1}$.

The identity map on $\csym$ gives two connections $\csym\circ\csym \to
\csym$ (using either $f_l$ or $f_r$ from Theorem \ref{thm:suffcover}).
This gives two new one-sided Hopf algebra structures on compositions.

\subsubsection{Hopf structure induced by $f_l$}
Let $p,q$ be composition trees and consider the product
$
   F_p \cdot F_q\ :=\ f_l(F_p)\star F_q\,.
$
At the level of indices in the $F$-basis, the connection $f_l$ sends the composition tree $p$ to the unique comb
$f_l(p)$ with the same number of vertices as $p$.
The action $f_l(F_p) \star F_q$ may be described as follows: 
split the comb $f_l(p)$ into a forest of $|q|{+}1$ combs in all possible ways;
graft each splitting onto the leaves of the forest in $q$;
comb the resulting forest of trees to get a new forest of combs.
For example,
$F_{1,3}\cdot F_{1,1}=F_{1,4}+F_{2,3}+F_{3,2}+F_{4,1}$, or alternatively,
 \begin{equation}\label{Eq:secondprodCC}
   F_{\;\includegraphics{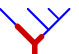}} \cdot   F_{\;\includegraphics{figures/p12.eps}}
   \ =\
   F_{\;\includegraphics{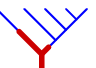}}\ +\ F_{\;\includegraphics{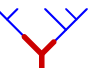}}\ +\
   F_{\;\includegraphics{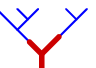}}\ +\ F_{\;\includegraphics{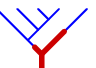}}\,.
 \end{equation}
This may be seen by unpainting and grafting the splittings~\eqref{Eq:comp_split}
onto the tree \includegraphics{figures/p12.eps}. Likewise, $F_{1,3}\cdot F_2=4F_{4}$, for no matter which of the four splittings of 
$f_l(1,3)$ 
is chosen, the grafting onto \includegraphics{figures/p21.eps} and subsequent combing will yield the same tree \includegraphics{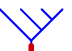}.

\subsubsection{Hopf structure induced by $f_r$}
Let $p,q$ be composition trees and consider the product
$
   F_p \cdot F_q\ :=\ F_p \star f_r(F_q)\,.
$
At the level of indices in the $F$-basis, the connection $f_r$ sends a composition tree
$q$ to the unique comb $f_r(q)$ with $|q|$ vertices. 
The action $F_p \star f_r(F_q)$ may be described as follows:  
first paint the comb $f_r(q)$;
next split the composition tree $p$ into a forest of $|q|{+}1$ composition trees in all
possible ways; finally, graft each forest onto the leaves of the painted tree $f_r(q)$ and
comb the resulting painted subtree (which comes from the nodes of $q$ and the painted
nodes of $p$). 
For example,
$F_{1,3}\cdot F_2=2F_{1,1,3}+F_{1,2,2}+F_{1,3,1}$, or alternatively,
 \begin{equation}\label{Eq:prodCC}
   F_{\;\includegraphics{figures/p3421.eps}} \cdot   F_{\;\includegraphics{figures/p21.eps}}
   \ =\
   F_{\;\includegraphics{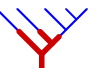}}\ +\ F_{\;\includegraphics{figures/p35421.eps}}\ +\
   F_{\;\includegraphics{figures/p35241.eps}}\ +\ F_{\;\includegraphics{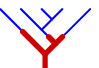}}\,.
 \end{equation}
This may be seen by grafting the splittings~\eqref{Eq:comp_split}
onto the tree $f_r(\includegraphics{figures/p21.eps}) = \includegraphics{figures/1.eps}$.

\subsection{Composition trees in topology}
A one-sided Hopf algebra $\delsym$ was defined
in~\cite[Section~7.3]{ForSpr:2010} whose $n$th graded piece had a basis
indexed by the faces of the ($n{-}1$)-dimensional simplex. 
We recount the product and coproduct introduced there. 
(The notation $\tilde{\delsym}$ was used
for this algebra in \cite{ForSpr:2010} to distinguish it from an algebra based only on the
vertices of the simplex.)  
Faces of the ($n{-}1$)-dimensional simplex correspond
to subsets $\setS$ of $[n]:=\{1,\dotsc,n\}$, so this is a Hopf algebra whose
$n$th graded piece also has dimension $2^n$, with fundamental basis
$F^{[n]}_{\setS}$. 

An ordered decomposition $n=p+q$ gives a splitting of $[n]$ into two pieces $[p]$ and
$\iota_p([q]):=\{p{+}1,\dotsc,n\}$.
Any subset $\setS\subseteq[n]$ gives a pair of subsets $\setS'\subseteq[p]$ and
$\setS''\subseteq[q]$,
\[
    \setS'\ :=\ \setS\cap[p]
   \qquad\mbox{and}\qquad
   \setS''\ :=\ \iota_p^{-1}(\setS\cap\{p{+}1,\dotsc,n\})\,.
\]
Then the coproduct is
\[
   \Delta(F^{[n]}_{\setS})\ =\
   \sum_{p+q=n} F^{[p]}_{\setS'}\otimes F^{[q]}_{\setS''}\,.
\]
For example, 
the coproduct on the basis element corresponding to $\{1\}\subseteq[3]$ is
\[
\Delta\bigl(F^{[3]}_{\{1\}}\bigr) \ = \ 
	F^{\emptyset}_\emptyset \otimes F^{[3]}_{\{1\}} +
	F^{[1]}_{\{1\}} \otimes F^{[2]}_{\emptyset} +
	F^{[2]}_{\{1\}} \otimes F^{[1]}_{\emptyset} +
	F^{[3]}_{\{1\}} \otimes F^{\emptyset}_{\emptyset} \,.
\]
This was motivated by constructions based on certain tubings of graphs.
In terms of tubings on an edgeless graph with three nodes, the coproduct takes the form 
\begin{equation}\label{Eq:DS_coprod}
  \Delta\raisebox{-2pt}{\includegraphics{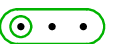}}\ =\
  \raisebox{-2pt}{\includegraphics{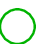}} \otimes
  \raisebox{-2pt}{\includegraphics{figures/1.3.g.eps}}  +
  \raisebox{-2pt}{\includegraphics{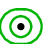}} \otimes
  \raisebox{-2pt}{\includegraphics{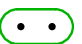}}  +
  \raisebox{-2pt}{\includegraphics{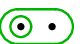}} \otimes
  \raisebox{-2pt}{\includegraphics{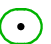}}  +
  \raisebox{-2pt}{\includegraphics{figures/1.3.g.eps}} \otimes
  \raisebox{-2pt}{\includegraphics{figures/0.0.g.eps}}\,.
\end{equation}
We leave it to the reader to make the identification (or see \cite{ForSpr:2010}).

The product $F^{[p]}_{\setS}\cdot F^{[q]}_{\setT}$ has one term for each shuffle of
$[p]$ with $\iota_p([q])$. 
The corresponding subset $\setR \subseteq [p{+}q]$ is the image of $[p]$ in the shuffle (not just $\setS$), together with the
image of $\setT$.
For example,
\begin{align}\label{Eq:delsim_prod}
  \raisebox{-3pt}{\includegraphics{figures/0.1.g.eps}}\cdot
  \raisebox{-3pt}{\includegraphics{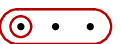}}&\ =\
	  \raisebox{-3pt}{\includegraphics{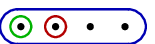}}\ +\
	  \raisebox{-3pt}{\includegraphics{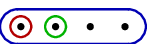}}\ +\
	  \raisebox{-3pt}{\includegraphics{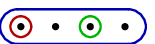}}\ +\ \vspace{-3pt}
	  \raisebox{-3pt}{\includegraphics{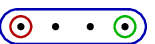}}\,,\\
\intertext{and}\vspace{-3pt}
  \raisebox{-3pt}{\includegraphics{figures/1.3.g.eps}}\cdot
  \raisebox{-3pt}{\includegraphics{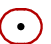}}&\ =\
	  \raisebox{-3pt}{\includegraphics{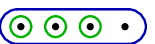}}\ +\
	  \raisebox{-3pt}{\includegraphics{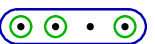}}\ +\
	  \raisebox{-3pt}{\includegraphics{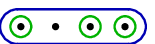}}\ +\
	  \raisebox{-3pt}{\includegraphics{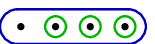}}\,. \nonumber
\end{align}

Let $\alpha$ be the bijection between subsets $\setS =\{
a,b,\dots,c,d\} \subseteq[n]$ and compositions
$\alpha(\setS)=(a,b{-}a,\dots,d-c,n{+}1{-}d)$ of $n{+}1$. Numbering the
nodes of a composition tree $1,\dotsc,n$ from left to right, the
subset of $[n]$ corresponding to the tree is comprised of the
colored nodes.
\[
  \begin{picture}(124,80)(0,-2)
   \put(  0,-2){\includegraphics{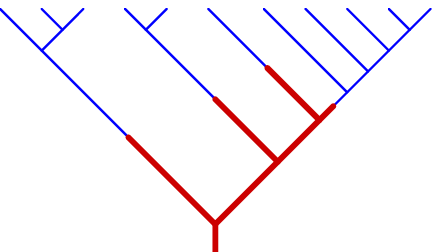}}
   \put(  4,70){1}  \put(15,70){2} \put(27,70){\Red{3}}
   \put( 39,70){4}  \put(51,70){\Red{5}} \put(65,70){\Red{6}}
   \put( 79,70){7}  \put(91,70){8} \put(104,70){9}
   \put(112,70){10}
  \end{picture}
  \raisebox{36pt}{$\longleftrightarrow\quad \{3,5,6\}\ \longleftrightarrow\ (3,2,1,4)\,.$}
\]

Applying this bijection to the indices of their fundamental bases
gives a linear isomorphism
$\balpha\colon\delsym\stackrel{\simeq}{\longrightarrow}\csym\circ\csym$.   
Comparing the definitions above,
this is clearly and isomorphism of coalgebras. 
Compare \eqref{Eq:comp_split} and \eqref{Eq:DS_coprod}. If we use
the second product on $\csym \circ \csym$ (induced by the connection $f_r$), then
$\balpha$ is nearly an isomorphism of the algebra, which can be
seen by comparing the examples~\eqref{Eq:prodCC}
and~\eqref{Eq:delsim_prod}. In fact, from the definitions given
above, it is an \demph{anti-isomorphism} ($\balpha(p\cdot
q)=\balpha(q)\cdot\balpha(p)$) of one-sided algebras. We may summarize
this discussion as follows.

\begin{theorem}
 The map $\balpha \colon \delsym \to (\csym \circ \csym,f_r)^{\mathrm{op}}$ is an isomorphism
 of one-sided Hopf algebras (with left-sided unit and right-sided antipode). 
\end{theorem}

\begin{corollary}\label{thm: F-S is cofree}
The one-sided Hopf algebra of simplices introduced in \cite{ForSpr:2010} is cofree as a coalgebra.
\end{corollary}


%
\end{document}
-------------------------------------------------------------------
\bibliographystyle{abbrv}

\bibliography{bibl}
\label{sec:biblio}

\end{document}